\documentclass[10pt,a4paper]{article}
\usepackage[utf8]{inputenc}
\usepackage[english]{babel}
\usepackage{hyperref}
\usepackage{amsmath,amsthm,enumerate}
\usepackage{amssymb}
\usepackage{tikz}
\usetikzlibrary{shapes}

\tikzstyle{main node}=[circle, draw,
                        inner sep=0pt, minimum width=15pt]
\tikzstyle{code node}=[circle, draw, fill=lightgray,
                        inner sep=0pt, minimum width=15pt]
\tikzstyle{shift node}=[circle, draw, fill=red,
                        inner sep=0pt, minimum width=15pt]
\tikzstyle{square main node}=[shape=rectangle, draw,
                        inner sep=7pt, minimum width=15pt]
\tikzstyle{square code node}=[shape=rectangle, draw, fill=lightgray,
                        inner sep=7pt, minimum width=15pt]
\tikzstyle{small node}=[circle, draw,
                        inner sep=0pt, minimum width=6pt]
\usepackage{subfigure}

\usepackage[hmargin=2.5cm,vmargin=3cm]{geometry}
\usepackage[color=green!30]{todonotes}

\usepackage{perpage} 
\MakePerPage{footnote} 

\newtheorem{theorem}{Theorem}

\newtheorem{proposition}[theorem]{Proposition}
\newtheorem{lemma}[theorem]{Lemma}

\newtheorem{corollary}[theorem]{Corollary}

\newtheorem{definition}[theorem]{Definition}

\newtheorem{remark}{Remark}

\newcommand{\LD}{\gamma^{\text{\tiny{L}}}}
\newcommand{\LDt}{\gamma_t^{\text{\tiny{L}}}}
\newcommand{\ID}{\gamma^{\text{\tiny{ID}}}}
\newcommand{\OLD}{\gamma^{\text{\tiny{OL}}}}
\newcommand{\IDt}{\gamma_t^{\text{\tiny{ID}}}}
\newcommand{\IDe}{\gamma_E^{\text{\tiny{ID}}}}
\newcommand{\IDs}{\gamma_S^{\text{\tiny{ID}}}}

\begin{document}

\title{Bounds and extremal graphs for\\total dominating identifying codes}
\author{Florent Foucaud\footnote{\noindent Université Clermont Auvergne, CNRS, Clermont Auvergne INP, Mines Saint-Etienne, LIMOS, 63000 Clermont-Ferrand, France.}~\footnote{This author was financed by the French government IDEX-ISITE initiative 16-IDEX-0001 (CAP 20-25) and by the ANR project GRALMECO (ANR-21-CE48-0004).}
\and Tuomo Lehtil\"a\footnote{Department of mathematics and statistics, University of Turku, Finland.  Research supported by the Academy of Finland grant 338797.}~\footnote{\noindent Univ Lyon, Universit\'e Claude Bernard, CNRS, LIRIS - UMR 5205, F69622 (France). Research supported by the Finnish Cultural Foundation.}
}

\maketitle

\begin{abstract}
An identifying code $C$ of a graph $G$ is a dominating set of $G$ such that any two distinct vertices of $G$ have distinct closed neighbourhoods within $C$. The smallest size of an identifying code of $G$ is denoted $\gamma^{\text{\tiny{ID}}}(G)$. 
When every vertex of $G$ also has a neighbour in $C$, it is said to be a total dominating identifying code of $G$, and the smallest size of a total dominating identifying code of $G$ is denoted by $\gamma_t^{\text{\tiny{ID}}}(G)$. 

Extending similar characterizations for identifying codes from the literature, we characterize those graphs $G$ of order $n$ with $\gamma_t^{\text{\tiny{ID}}}(G)=n$ (the only such connected graph is $P_3$) and $\gamma_t^{\text{\tiny{ID}}}(G)=n-1$ (such graphs either satisfy $\gamma^{\text{\tiny{ID}}}(G)=n-1$ or are built from certain such graphs by adding a set of universal vertices, to each of which a private leaf is attached).

Then, using bounds from the literature, we remark that any (open and closed) twin-free tree of order $n$ has a total dominating identifying code of size at most $\frac{3n}{4}$. This bound is tight, and we characterize the trees reaching it. Moreover, by a new proof, we show that this upper bound actually holds for the larger class of all twin-free graphs of girth at least~5. The cycle $C_8$ also attains the upper bound. We also provide a generalized bound for all graphs of girth at least~5 (possibly with twins).

Finally, we relate $\gamma_t^{\text{\tiny{ID}}}(G)$ to the similar parameter $\gamma^{\text{\tiny{ID}}}(G)$ as well as to the location-domination number of $G$ and its variants, providing bounds that are either tight or almost tight.
\end{abstract}

\noindent\textbf{Keywords:} Identifying codes; total dominating sets; extremal problem; upper bound
\section{Introduction}

An identifying code of a graph is a dominating set that allows distinguishing all pairs of vertices by means of their neighbourhoods within the identifying code. This extensively studied concept is related to other similar notions that deal with domination-based identification of the vertices/edges of a graph or hypergraph, such as locating-dominating sets~\cite{S87}, separating systems~\cite{BS07,B72}, discriminating codes~\cite{CCCHL08}, or test covers~\cite{MS85}. This class of problems has applications in fault-detection in networks~\cite{KCL98,UTS04}, biological diagnosis~\cite{MS85} and machine learning~\cite{BGL19}, to name a few. A total dominating set is a set $D$ of vertices such that every vertex has a neighbour in $D$. The concept of a total dominating set is perhaps the most studied alternative variant in the field of graph domination, see the dedicated book~\cite{bookTD} on this topic.

In this paper, we study total dominating identifying codes, that are sets of vertices that are both identifying codes and total dominating sets. Our focus is on upper bounds and extremal graphs for the smallest size of such a set, as well as bounds involving other related concepts.


\paragraph{Notations and definitions.} In this paper, we consider finite undirected graphs. We first define some basic notations. A vertex $u\in V(G)$ is said to be a \textit{leaf}, if it has degree exactly~1. A vertex $v\in V(G)$ is said to be a \textit{support vertex} if it has an adjacent leaf.
We denote by $L(G)$ the set of leaves and by $S(G)$ the set of support vertices in graph $G$. Moreover, we denote the number of leaves and support vertices by $\ell(G)=|L(G)|$ and $s(G)=|S(G)|$, respectively. The \textit{girth} of a graph is the smallest length of one of its cycles. 

We denote by $N(v)\subseteq V(G)$ the \textit{open neighbourhood} of vertex $v$ and by $N[v]=N(v)\cup\{v\}$, its \textit{closed neighbourhood}. If $C$ is a set of vertices, or a \textit{code}, and $v$, a vertex, we denote the intersection between $N[v]$ and code $C$ by the $I$-set of $v$, $I(v)=N[v]\cap C$. \textit{Identifying codes} were defined over twenty years ago in~\cite{KCL98} by Karpovsky et al. and since then they (and related concepts) have been studied in a large number of articles, see~\cite{biblio} for an online bibliography. A set $C\subseteq V(G)$ is called a \emph{separating code} of $G$ if for each pair of distinct vertices $u,v\in V(G)$, their $I$-sets are distinct, that is, $$I(u)\neq I(v).$$

An \emph{identifying code} of $G$ is a set of vertices that covers every vertex $v$, that is, $I(v)\neq \emptyset$, and is a separating code. (Note that every separating code is ``almost'' an identifying code, as at most one vertex may remain uncovered by the separating code.) A \emph{total dominating identifying code} is a separating code that is also a total dominating set, that is, every vertex of the graph has a neighbour in the code. Any total dominating identifying code is also an identifying code.

The vertices of a code are called \emph{codewords}. A codeword $x$ is said to \emph{separate} two vertices if it belongs to the closed neighbourhood of exactly one of them. We also say that the codeword $x$ separates vertex $u$ from vertex $v$ or vice versa meaning that codeword $x$ separates these two vertices. (We sometimes use \textit{distinguish} as a synonym of separate.) Two vertices are \emph{open twins} if their open neighbourhoods are the same, and \emph{closed twins} if their closed neighbourhoods are the same. A graph admits a separating code if and only if it has no pairs of closed twins~\cite{KCL98}; in that case we say the graph is \emph{identifiable}. We say that a graph is \textit{twin-free} if it contains neither open nor closed twins. Twins are important for (total dominating) identifying codes, indeed closed twins cannot be separated, and for any set of mutually open twins, at most one can be absent from any separating code. A graph admits a total dominating set if and only if its minimum degree is at least~1. For an identifiable graph $G$, we denote by $\ID(G)$ the smallest size of an identifying code of $G$. In the context of total dominating identifying codes, by saying a graph is \emph{identifiable} we  also assume implicitly that it admits a total dominating set. For such an identifiable graph $G$, we denote by $\IDt(G)$ the smallest size of a total dominating identifying code of $G$. Total dominating identifying codes have been studied only in a handful of papers, see~\cite{C08,HHH06,NLG16,OCR14a,OCR14,PGP21,JR18}.

Identifying codes have sometimes been called \emph{differentiating-dominating sets} in the literature, see for example~\cite{GGNUV01}. Total dominating identifying codes have been called \emph{differentiating-total dominating sets}, however due to the now standard term of ``identifying code'' we believe, it is a better choice to call them \emph{total dominating identifying codes}, thus we do so in this paper.

\paragraph{Further related concepts.} Besides identifying and total dominating identifying codes, quite many other related concepts have been studied. We present the relationships between some of these different types of dominating and locating codes in connected graphs in Figure~\ref{Dominating codes}.  As one can see on the figure, total dominating identifying codes are directly related to several important concepts in the area.

A set $C$ is \textit{locating-dominating} if we have $I(u)\neq I(v)$ for each distinct $u,v\not \in C$~\cite{S87}. Furthermore, set $C$ is \textit{locating-total dominating} if it is locating-dominating and total dominating~\cite{HHH06}. A code is \textit{self-identifying} if for any distinct $u,v$ we have $I(u)\setminus I(v)\neq \emptyset$~\cite{JL20}. Self-identifying codes have also been studied as \textit{$(1,\leq1)^+$-identifying codes}~\cite{HL07}. Denote by $I(X)=\bigcup_{u\in X}I(u)$ where $X$ is a set of vertices. Code $C$ is a \textit{$(1,\leq 4)$-identifying code} if for any distinct sets $X,Y$ with $|X|,|Y|\leq 4$ we have $I(X)\neq I(Y)$~\cite{HL07,KCL98}. Moreover, code $C$ is an \textit{error-correcting identifying code} if $|I(u)|\geq3$ for each vertex $u$ and $|I(u)\triangle I(v)|\geq3$ for any distinct vertices $v$ and $u$~\cite{JS22,SS18}. Finally, set $C$ is an \textit{open (neighbourhood) locating-dominating} if we have $N(v)\cap C\neq \emptyset$ for each vertex $v$ and for each distinct pair of vertices $u, v$, we have $N(v)\cap C\neq N(u)\cap C$~\cite{HLS02, SS10}.

Each arc in Figure~\ref{Dominating codes} follows trivially from the above definitions, with the possible exception of the arc from SID to TID and the arcs adjacent to OLD. Assume that $C$ is a self-identifying code in graph $G$ which also admits a total dominating identifying code. If for any $c\in C$ we have $I(c)=\{c\}$ and $u\in N(c)$, then $I(c)\subseteq I(u)$, a contradiction. Thus, $C$ is total dominating and it is identifying by definition. Then, consider the arc from OLD to TLD. Let $C$ be an open locating-dominating set in $G$. Then, $N(v)\cap C\neq \emptyset$ and thus, $C$ is total dominating. Moreover, we have $I(v)\neq I(u)$ for each pair of distinct pair of non-codewords $u$ and $v$. Then, consider the arc from EID to OLD. Let $C$ be an error-correcting-identifying code in connected graph $G$. If $I(v)=\{v\}$ and $u\in N(v)$, then $I(u)=I(\{u,v\}$. Thus, $C$ is total dominating. Moreover, if $N(u)\cap C=N(v)\cap C$, then $I(u)\triangle I(v)\subseteq \{v,u\}$, a contradiction.

The cardinality of an optimal locating-dominating sets in graph $G$ is denoted by $\LD(G)$. Similarly, we use $\LDt(G)$ for locating-total dominating sets, $\OLD(G)$ for open-locating-dominating sets,  $\IDe(G)$ for error-correcting identifying codes and $\IDs(G)$ for self-identifying codes.

\begin{figure}[ht!]
\centering
\scalebox{0.8}{\begin{tikzpicture}

    \node[main node, minimum size=0.85cm] (2)                                     {LD};
    \node[main node, minimum size=0.85cm] (4) [above left = 1.03cm and 1.2cm of 2]{ID};
    \node[main node, minimum size=0.85cm] (6) [above = 2.45cm of 4]                {SID};
    \node[main node, minimum size=0.85cm, fill=gray!50] (7) [above = 2.45cm of 2]               {TID};
    \node[main node, minimum size=0.85cm] (8) [right = 2.8cm of 6]                {EID};
    \node[main node, minimum size=0.85cm] (9) [right = 2.8cm of 2]                {TD};    
    \node[main node, minimum size=0.85cm] (10)[below = 5.8cm of 8]                {D};
    \node[main node, minimum size=0.85cm] (11)[above = 2.45cm of 10]               {TLD};
    \node[main node, minimum size=0.85cm] (12)[above = 2.45cm of 7]               {4ID};
    \node[main node, minimum size=0.85cm] (13) [right = 2.8cm of 7]               {OLD};
    

    \path[draw,thick]
    (2) edge [->]node {} (10)
    (11) edge [->]node {} (2)
    (8) edge [->]node {} (7)
    (7) edge [->]node {} (4)
    (7) edge [->]node {} (11)
    (11) edge [->]node {} (9)
    (6) edge [->]node {} (7)
    (4) edge [->]node {} (2)
	(9) edge [->]node {} (10)
	(12) edge [->]node {} (8)
	(12) edge [->]node {} (6)
	(13) edge [->]node {} (11)
	(8) edge [->]node {} (13)
;
\end{tikzpicture}}
\caption{Relations between some types of dominating sets in connected graphs. An arrow from $X$ to $Y$ denotes that each code of type $X$ in a graph is also a code of type $Y$ (when that graph admits codes of both types). The gray node corresponds to total dominating identifying codes, the main focus of the paper. $D$ stands for dominating set, $TD$ stands for total dominating set, $LD$ stands for locating-dominating set, $TLD$ stands for locating-total dominating set, $OLD$ stands for open-locating-dominating set, $ID$ stands for identifying code, $TID$ stands for total dominating identifying code, $EID$ stands for error-correcting identifying code, $SID$ stands for self-identifying code and $4ID$ stands for $(1,\leq4)$-identifying code.}\label{Dominating codes}
\end{figure}
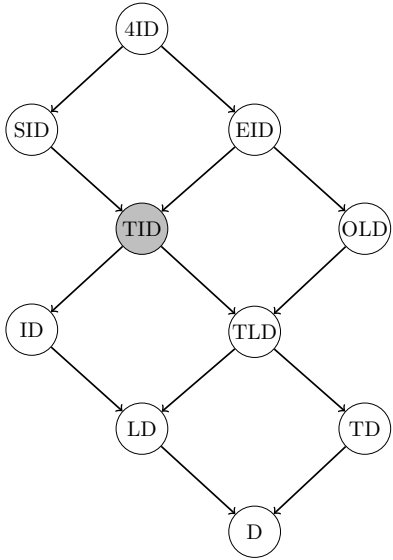

The computational problem associated with determining $\IDt(G)$ for an input graph $G$ is NP-hard, and has been studied in~\cite{PGP21}. Lower and upper bounds for parameter $\IDt$ in trees have been proved in~\cite{C08,HHH06,NLG16,JR18}. Different graph classes, in particular graph products, were studied in~\cite{OCR14a,OCR14}.

\paragraph{Our results.} Our main result is to characterize those graphs $G$ of order $n$ for which $\IDt(G)\geq n-1$. We show that the only connected graph $G$ with $\IDt(G)=n$ is the 3-vertex path $P_3$. The graphs $G$ for which $\IDt(G)=n-1$ form a rich graph class. This class of graphs includes those graphs for which $\ID(G)=n-1$, characterized in~\cite{FGKNPV11} as essentially (1) stars, (2) the complements of half-graphs, and (3) graphs built from any number of graphs from (2) using complete join operations and potentially, the addition of a single universal vertex. We show that besides these examples, one can obtain a graph $G$ with $\IDt(G)\geq n-1$ from a graph from (2) or (3) or the empty graph, by completely joining it to a copy of $K_m$ (for any integer $m\geq 1$), and add a private leaf to each vertex of $K_m$. We then show that these cases are essentially the only possibilities to obtain an extremal graph for parameter $\IDt$.

All the graphs in the above constructions either have many twins, or have (many) short cycles. We show that in the absence of these two obstructions, one can obtain an upper bound on $\IDt$ significantly smaller than $n$. Indeed, we first notice that two bounds from the literature imply that every twin-free tree $T$ of order $n$ satisfies $\IDt(T)\leq 3n/4$, and by a new proof, we generalize this upper bound to all identifiable graphs of girth at least~5. The bound is shown to be tight for certain trees, and for the cycle $C_8$.

Finally, we study the ratio between parameter $\IDt$ and related parameters; natural lower bounds for $\IDt(G)$ are $\LD(G)$, $\LDt(G)$ and $\ID(G)$, as we can see from Figure \ref{Dominating codes}. We show that for any identifiable graph $G$, $\IDt(G)\leq 2\ID(G)-2$ and $\IDt(G)\leq 2\LDt(G)$ (both bounds are tight). Interestingly, we can show that $\IDt(G)\leq 3\LD(G)-\log_2(\LD(G)+1)$, and we show the bound is nearly tight, as there are infinitely many connected graphs $G$ for which $\IDt(G)=3\LD(G)-2\log_2(\LD(G)+1)$. Moreover, we also show that without restricting the class of graphs, neither $\IDs(G)$ nor $\IDe(G)$ gives a useful upper bound for $\IDt(G)$. In other words, there are graphs for which $\IDt(G)$ is much smaller than either of $\IDs(G)$ or $\IDe(G)$.

We present our characterization of extremal graphs in Section~\ref{sec:extremal}. The bound for twin-free graphs of girth at least~5 is presented in Section~\ref{sec:twinfree}. The bounds relating $\IDt$ to related parameters are presented in Section~\ref{sec:bounds-related}. We conclude in Section~\ref{sec:conclu}.

\paragraph{Further related work.} Our results were inspired by the related work below.

Characterizations of extremal graphs for identifying codes and related parameters were studied in several papers, for example for locating-dominating sets~\cite{CHMPP12,CHL07}, for identifying codes~\cite{CHL07,FGKNPV11}, for open neighbourhood locating-dominating sets~\cite{FGRS21}, and for discriminating codes~\cite{CCCHL08}.

Bounds for twin-free graphs have been studied for related graph parameters. It was proved in~\cite{ourIDpaper} that for every twin-free bipartite graph $G$ of order $n$, $\ID(G)\leq 2n/3$, and the bound is tight exactly for 2-coronas of bipartite graphs (that is, bipartite graphs $B$ for which a private copy of $P_2$ is attached to each vertex of $B$ by one of the ends of $P_2$). It was proved in~\cite{GGM14} that every twin-free bipartite graph and every twin-free graph with no $4$-cycles has a locating-dominating set of size at most $n/2$; the bound is tight for infinitely many trees, which are characterized in~\cite{Heia}. In~\cite{FH16}, it was proved that every twin-free graph with no $4$-cycle has a locating-total dominating set of size at most $2n/3$. It is conjectured that these two bounds hold for all twin-free graphs~\cite{FH16,GGM14}.

Bounds for graphs of girth at least~5 were given for identifying codes in~\cite{BFH15,ourIDpaper,FP11}. In particular, generalizing a result from~\cite{BFH15}, it is shown in~\cite{ourIDpaper} that for every graph $G$ of order $n$ and girth at least~5, we have $\ID(G)\leq\frac{5n+2\ell(G)}{7}$, a bound which is tight.

Relations between identification-type graph parameters were provided in~\cite{GKM08} (locating-dominating sets and identifying codes) and~\cite{SewellPhD} (locating-dominating sets, identifying codes, and open-locating-dominating sets). It is shown that for every graph $G$, any two of these parameters' values cannot be more than a factor~2 apart from each other. Such bounds do not seem to be known for total dominating identifying codes, however, in \cite[Theorem 2.3]{OCR14}, infinitely many graphs $G$ satisfying $\IDt(G)=\tfrac{3}{2}\ID(G)$ are constructed.

\section{Characterizing graphs with largest possible total dominating identifying codes}\label{sec:extremal}

In this section, we characterize the graphs which attain extremal values for total dominating identifying codes. 

\subsection{Preliminaries}

One can easily check that no graph of order at most~2 admits a total dominating identifying code, since $P_1$ has no total dominating set and $P_2$ is not identifiable. We start by showing that $P_3$ is the only connected identifiable graph of order $n$ whose smallest total dominating identifying code has size $n$.

\begin{proposition}\label{Prop:IDtn-1}
If $G$ is a connected identifiable graph of order $n$ (thus $n\geq 3$), then we have $\IDt(G)\leq n-1$, unless $G$ is $P_3$ (and $\IDt(P_3)=3$).
\end{proposition}
\begin{proof}
Since $P_3$ is the only identifiable graph of order at most~$3$ admitting a total dominating set and $\IDt(P_3)=3$, we may assume $n\geq 4$. It is known that for any identifiable graph $G$ with at least one edge, there is always a vertex $x$ such that $V(G)\setminus\{x\}$ is an identifying code of $G$, see~\cite{GM07}. Moreover, $V(G)\setminus\{x\}$ is a total dominating set, unless $x$ is a support vertex. Thus, if there are no support vertices in $G$, we are done. Otherwise, let $x$ be a support vertex of $G$ and let $y$ be a leaf neighbour of $x$. Since $V(G)\setminus\{y\}$ is a total dominating set, if $V(G)\setminus\{y\}$ is also an identifying code, then we are done. Otherwise, there must exist two vertices $u,v$ of $G$ that can only be distinguished by $y$, that is, such that $N[u]=N[v]\cup\{y\}$. If $y\in\{u,v\}$, then $y=u$ and $v=x$, but this is not possible since $y$ separates $u$ and $v$, a contradiction. Hence, $u$ is a neighbour of $y$, that is, $u=x$. Since $n\geq 4$, $u$ and $v$ have a common neighbour, say $w$. As $v$ is not a support vertex, the set $V(G)\setminus\{v\}$ is a total dominating set. We claim that it is also an identifying code of $G$, which would prove the claim. Indeed, any pair $s,t$ of vertices with $y\notin\{s,t\}$ such that $v$ separates $s$ from $t$, is also separated by $u$, and $y$ is separated from $x$ by $w$ (and from every other vertex by itself).
\end{proof}

The authors of~\cite{HHH06} characterized the trees $T$ of order $n\geq 4$ with $\IDt(T)=n-1$ to be exactly the set of stars, and $P_4$. The set of graphs $G$ of order $n\geq 4$ with $\IDt(G)=n-1$ necessarily contains all those graphs without isolated vertices for which $\ID(G)=n-1$ (except $P_3$). The graphs $G$ of order $n$ with $\ID(G)=n-1$ were characterized in~\cite{FGKNPV11}, based on the following graph families.

\begin{definition}[\cite{FGKNPV11}]
For any non-negative integer $k$, we define the graph $A_k$ of order $2k$ as the graph on vertex set $\{x_1,\ldots,x_{2k}\}$ where $x_i$ is adjacent to $x_j$ if and only if $|i-j|\leq k-1$.

We denote by $\mathcal{A}$ the set of graphs obtained by taking any number (possibly, zero) of disjoint copies of graphs in the family $\{A_k~|~k\geq 1\}$ and joining every pair of these graphs by all possible edges between them. We denote by $\mathcal{A}^*$ the set $\mathcal{A}$ without the graphs $A_0$ and $A_1$.

For two graphs $G$ and $H$, we denote by $G\bowtie H$ the \emph{complete join} of $G$ and $H$, that is, the graph obtained from a copy of $G$ and a copy of $H$ by adding all possible edges between the two copies. For a set $\mathcal{C}$ of graphs, we denote by $\mathcal{C}\bowtie K_1$ the set of graphs $\{G\bowtie K_1\mid G\in \mathcal{C}\}$. 
\end{definition}

Note that $A_0$ is the empty graph, $A_1$ is the edgeless graph of order~2, and $A_2$ is the 4-vertex path. For $k\geq 2$, the graph $A_k$ is isomorphic to the $(k-1)$-th power of the path $P_{2k}$ and can be partitioned into two cliques, as shown in Figure~\ref{fig:A_k}. In fact $A_k$ is the complement of the \emph{half-graph} of order $2k$ (half-graphs form a family of special bipartite graphs defined by Erd\H{o}s and Hajnal, see~\cite{EH84}).

An example of a graph in $\mathcal{A}$ is provided in Figure~\ref{fig:extremal-examples}. As a special case, the set $\mathcal{A}$ contains all even-order complete graphs minus a maximum matching (by considering only copies of $A_1$ in the construction), and by the addition of a universal vertex, $\mathcal{A}\bowtie K_1$ contains all odd-order complete graphs minus a maximum matching.

    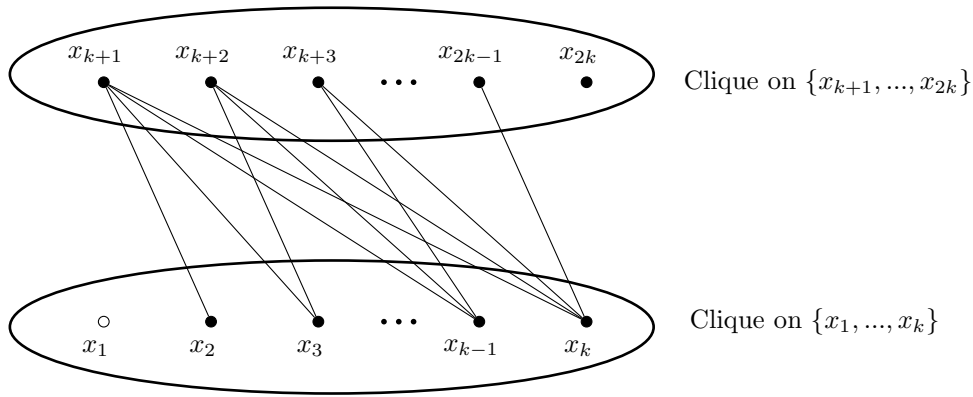
\begin{figure}[ht]
        \begin{center}
\scalebox{1}{\begin{tikzpicture}[join=bevel,inner sep=0.5mm,]
  \node (x_kplus1) at (90bp,200bp) [draw, circle, fill=black] {};
  \node  at (87bp,210bp) {$x_{k+1}$};
  \node (x_kplus2) at (130bp,200bp) [draw, circle, fill=black] {};
  \node  at (127bp,210bp) {$x_{k+2}$};
  \node (x_kplus3) at (170bp,200bp) [draw, circle, fill=black] {};
  \node  at (167bp,210bp) {$x_{k+3}$};
  \node[scale=2] at (200bp,200bp) {$...$};
  \node (x_2kminus1) at (230bp,200bp) [draw, circle, fill=black] {};
  \node  at (227bp,210bp) {$x_{2k-1}$};
  \node (x_2k) at (270bp,200bp) [draw, circle, fill=black] {};
  \node  at (267bp,210bp) {$x_{2k}$};
  \node (x1) at (90bp,110bp) [draw, circle] {};
  \node  at (87bp,100bp) {$x_1$};
  \node (x2) at (130bp,110bp) [draw, circle, fill=black] {};
  \node  at (127bp,100bp) {$x_2$};
  \node (x3) at (170bp,110bp) [draw, circle, fill=black] {};
  \node  at (167bp,100bp) {$x_3$};
  \node[scale=2] at (200bp,110bp) {$...$};
  \node (x_kminus1) at (230bp,110bp) [draw, circle, fill=black] {};
  \node  at (227bp,100bp) {$x_{k-1}$};
  \node (x_k) at (270bp,110bp) [draw, circle, fill=black] {};
  \node  at (267bp,100bp) {$x_k$};
  \draw [-] (x_kplus1) -- node[above,sloped] {} (x_k);
  \draw [-] (x_kplus1) -- node[above,sloped] {} (x2);
  \draw [-] (x_kplus1) -- node[above,sloped] {} (x3);
  \draw [-] (x_kplus1) -- node[above,sloped] {} (x_kminus1);
  \draw [-] (x_kplus2) -- node[above,sloped] {} (x_kminus1);
  \draw [-] (x_kplus2) -- node[above,sloped] {} (x_k);
  \draw [-] (x_kplus2) -- node[above,sloped] {} (x3);
  \draw [-] (x_kplus3) -- node[above,sloped] {} (x_kminus1);
  \draw [-] (x_kplus3) -- node[above,sloped] {} (x_k);
  \draw [-] (x_2kminus1) -- node[above,sloped] {} (x_k);
  \node at (360bp,200bp) {Clique on $\{x_{k+1},...,x_{2k}\}$};
\draw[line width=1pt, draw = black] (175bp,203bp) ellipse(120bp and 25bp) node{};
  \node at (355bp,110bp) {Clique on $\{x_1,...,x_k\}$};
  \draw[line width=1pt, draw = black] (175bp,108bp) ellipse(120bp and 25bp) node{};
  \end{tikzpicture}}
  \end{center}
  \caption{The graph $A_{k}$ of order $n=2k$ has total dominating identifying code number $n-1$ (the black vertices form an optimal total dominating identifying code).}
\label{fig:A_k}
\end{figure}

    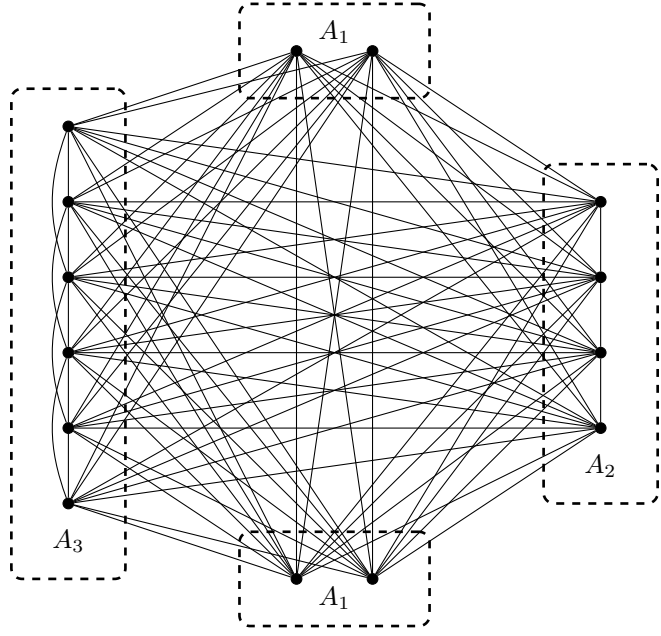
\begin{figure}[ht]
        \begin{center}
\scalebox{1}{\begin{tikzpicture}[join=bevel,inner sep=0.5mm,]
    \node (x1) at (0,0) [draw, circle, fill=black] {};
    \node (x2) at (0,1) [draw, circle, fill=black] {};
    \node (x3) at (0,2) [draw, circle, fill=black] {};
    \node (x4) at (0,3) [draw, circle, fill=black] {};
    \node (x5) at (0,4) [draw, circle, fill=black] {};
    \node (x6) at (0,5) [draw, circle, fill=black] {};
    \draw [-] (x1) -- (x2) -- (x3) -- (x4) -- (x5) -- (x6);
    \draw (x1) to[bend left=20] (x3);
    \draw (x2) to[bend left=20] (x4);
    \draw (x3) to[bend left=20] (x5);
    \draw (x4) to[bend left=20] (x6);
    \node[rectangle, line width=1pt, draw = black, dashed, rounded corners, minimum width=1.5cm, minimum height=6.5cm] (A3) at (0,2.25) {};
    \node (A3name) at (0,-0.5) {$A_3$};

    \node (y1) at (7,1) [draw, circle, fill=black] {};
    \node (y2) at (7,2) [draw, circle, fill=black] {};
    \node (y3) at (7,3) [draw, circle, fill=black] {};
    \node (y4) at (7,4) [draw, circle, fill=black] {};

    \draw [-] (y1) -- (y2) -- (y3) -- (y4);
    \node[rectangle, line width=1pt, draw = black, dashed, rounded corners, minimum width=1.5cm, minimum height=4.5cm] (A2) at (7,2.25) {};
    \node (A2name) at (7,0.5) {$A_2$};
    
    \node (a1) at (3,6) [draw, circle, fill=black] {};
    \node (a2) at (4,6) [draw, circle, fill=black] {}; 
    \node (b1) at (3,-1) [draw, circle, fill=black] {};
    \node (b2) at (4,-1) [draw, circle, fill=black] {}; 
    \node[rectangle, line width=1pt, draw = black, dashed, rounded corners, minimum width=2.5cm, minimum height=1.25cm] (A1) at (3.5,-1) {};
    \node (A1name) at (3.5,-1.25) {$A_1$};    
    \node[rectangle, line width=1pt, draw = black, dashed, rounded corners, minimum width=2.5cm, minimum height=1.25cm] (A1b) at (3.5,6) {};
    \node (A1bname) at (3.5,6.25) {$A_1$};    
    
    \foreach \i in {1,2}\foreach \j in {1,2} \draw [-] (a\i)--(b\j);
    \foreach \i in {1,2,3,4,5,6}\foreach \j in {1,2,3,4} \draw [-] (x\i)--(y\j);
    \foreach \i in {1,2}\foreach \j in {1,2,3,4,5,6} \draw [-] (a\i)--(x\j) (b\i)--(x\j);
    \foreach \i in {1,2}\foreach \j in {1,2,3,4} \draw [-] (a\i)--(y\j) (b\i)--(y\j);  
  
  \end{tikzpicture}}
  \end{center}
  \caption{A graph in $\mathcal A^*$ built from two copies of $A_1$ and from one copy of each of $A_2$ and $A_3$.}
\label{fig:extremal-examples}
\end{figure}

For completeness, we give a proof for the following result from~\cite{FGKNPV11}, initially stated for usual identifying codes but which also holds for total dominating identifying codes.

\begin{proposition}[\cite{FGKNPV11}]\label{prop:extremal-family-IDtn-1}
For every graph $G$ of order $n$ in $\mathcal{A}^*\cup(\mathcal{A}^*\bowtie K_1)$, every separating code has size at least $n-1$, and $\IDt(G)=n-1$. Moreover, if $G\in \mathcal{A}^*\bowtie K_1$, then the only separating code is $V(G)$ minus the unique universal vertex. 
\end{proposition}
\begin{proof}
First, assume that $k\geq 2$, we show that $\IDt(A_k)=2k-1$. For every $i$ with $1\leq i\leq k-1$, $x_{k+i}$ is the only vertex separating $x_i$ from $x_{i+1}$ and similarly, $x_{k-i+1}$ is the only vertex separating $x_{2k-i}$ and $x_{2k-i+1}$; thus, all of $x_2,\ldots,x_{2k-1}$ must belong to any separating code of $A_k$. Finally, $x_{k}$ and $x_{k+1}$ can only be separated by one of $x_1$ and $x_{2k}$. Thus, any separating code has size at least $2k-1=n-1$, and so $\IDt(G)\geq n-1$. The set $V(A_k)\setminus\{x_1\}$ is a total dominating identifying code of size $n-1$.

Now, consider any graph in $\mathcal{A}^*$ and number the $\ell\geq 1$ copies of graphs in $\{A_k~|~k\geq 1\}$ as $A^1_{i_1},\ldots,A^\ell_{i_\ell}$ where $A^s_{i_s}$ is a copy of $A_{i_s}$ whose vertices are labeled $x_1^s,\ldots,x^s_{2i_s}$. Consider any copy $A^s_{t}$ (where $t=i_s$) of $A_t$ from the construction of $G$. If $t\geq 2$, by the same arguments as above, we see that all the vertices $x^s_2,\ldots,x^s_{2t-1}$ from $A^s_t$ must belong to any separating code of $G$ (since all vertices of $A^s_t$ have the same neighbourhoods outside of $A^s_t$), and at least one of $x^s_{1}$ and $x^s_{2t}$ in $A^s_t$ must belong to the code in order to separate $x^s_t$ from $x^s_{t+1}$. If $t=1$, one of $x^s_1$ and $x^s_2$ necessarily belongs to the code since these two vertices are open twins in $G$. Without loss of generality, by the symmetries of $x^s_{1}$ and $x^s_{2t}$, we assume that $x^s_1$ belongs to the code, and we also assume that $x^s_1$ belongs to the code for each copy $A^s_{i_s}$ in $G$ ($1\leq s\leq\ell$). Now, for any pair of copies $A^s_{i_s}$ and $A^t_{i_t}$ in $G$, notice that $x^s_{i_s}$ and $x^t_{i_t}$ can only be separated by one of $x^s_{2i_s}$ and $x^t_{2i_t}$. Hence, at most one vertex of type $x^j_{2i_j}$ in some $A^j_{i_j}$ can be omitted from any separating code, and so any separating code has size $n-1$. Note that $V(G)\setminus\{x^j_{2i_j}\}$ is a total dominating identifying code of size $n-1$.

Assume now that $G\in \mathcal{A}^*\bowtie K_1$ and let $u$ be the universal vertex from the copy of $K_1$ in $G$. We use the same reasoning, to show that all vertices except possibly $u$ and some $x^j_{2i_j}$ must belong to the separating code. However, to separate $u$ from $x^j_{i_j}$, we must also include vertex $x^j_{2i_j}$, and so again any separating code has size at least $n-1$, and $\IDt(G)\geq n-1$. Moreover, $V(G)\setminus\{u\}$ is a total dominating identifying code of size $n-1$.
\end{proof}

The following characterization was proved in~\cite{FGKNPV11}.

\begin{theorem}[\cite{FGKNPV11}]\label{The:IDn-1}
If $G$ is a connected identifiable graph on $n$ vertices, then $\ID(G)=n-1$ if and only if $G\in \{K_{1,t}\mid t\geq2\}\cup \mathcal{A}^*\cup(\mathcal{A}^*\bowtie K_1).$
\end{theorem}

\subsection{The characterization}

We next state our characterization theorem, which we will prove after some preliminary lemmas.

\begin{theorem}\label{The:IDtn-1Char}
For any connected graph $G$ on $n\geq3$ vertices, we have $\IDt(G)\geq n-1$ if and only if either:
\begin{itemize}
    \item[(i)] $\ID(G)\geq n-1$, that is $G\in  \{K_{1,t}\mid t\geq2\}\cup \mathcal{A}^*\cup(\mathcal{A}^*\bowtie K_1)$, or 
    \item[(ii)] $G'=G''\bowtie K_m$, where $m\geq1$ and $G''\in \mathcal{A}\cup(\mathcal{A}\bowtie K_1)$, and $G$ is obtained from $G'$ by attaching a leaf to each vertex in the clique $K_m$.
    \end{itemize}
    Moreover, $\IDt(G)= n$ if and only if $G=P_3$.
\end{theorem}

Next, we prove that the family of graphs described in Theorem~\ref{The:IDtn-1Char}(ii) (whose members have an identifying code of size less than $n-1$), indeed is extremal for total dominating identifying codes.

\begin{proposition}\label{prop:tID-new-extremal}
If $G'=G''\bowtie K_m$, where $m\geq1$ and $G''\in \mathcal{A}\cup(\mathcal{A}\bowtie K_1)$, and the graph $G$ of order $n\geq 3$ is obtained from $G'$ by attaching a leaf to each vertex in the clique $K_m$, then $\IDt(G)\geq n-1$.
\end{proposition}
\begin{proof}
Let $G$ be obtained from $G'$ and $G''$ as described in the statement (note that possibly, $G''$ is the empty graph, the graph of order~1, or the edgeless graph of order~2).

Let $C$ be an optimal total dominating identifying code in $G$. Observe that for it to be total dominating, every vertex in $K_m$ has to be in $C$. Moreover, to separate vertices in the clique $K_m$, at least $m-1$ of the leaves must be in the code. Furthermore, none of the vertices in the clique $K_m$ separate vertices in $G''$ from each other. Thus, $C\cap V(G'')$ must be a separating code of $G''$, but in any separating code of $G''$, there is at most one non-codeword by Proposition~\ref{prop:extremal-family-IDtn-1}. 
Thus, we have $|C|\geq n-2$, and the two vertices not yet fixed to be in $C$ are a leaf and a vertex of $G''$ that can be omitted from a separating code of $G''$.

Assume now that $|C|=n-2$, $w$ is the non-codeword in $G''$ and $u$, which is adjacent to $v$, is the non-codeword leaf in $G$. Observe that if $G''\in \mathcal{A}\bowtie K_1$, then $w$ is the universal vertex in $G''$ and hence, $I(w)=I(v)$, a contradiction. Moreover, if $G''\in \mathcal{A}\setminus\{A_0,A_1\}$, then the non-codeword corresponds to a vertex of type $x_1$ or $x_{2i_j}$ in some subgraph $A_{i_j}$, say, $x_1$. However, now $I(x_{2i_j-1})=I(v)$, again a contradiction. Therefore, $C$ has cardinality $n-1$. By the same arguments, when $|V(G'')|\leq 2$, we also have $\IDt(G)\geq n-1$.
\end{proof}

Some example graphs $G$ of order $n$ for which $\IDt(G)=n-1$ but $\ID(G)<n-1$ are depicted in Figure~\ref{fig:extremal-examples-tID}.

\begin{figure}[ht]
\begin{center}

  \subfigure[The graph obtained from $K_m$ ($m\geq 2$) by attaching a leaf to each vertex.]{
    \scalebox{1}{\begin{tikzpicture}[join=bevel,inner sep=0.5mm]

        \node (x1) at (0,0) [draw, circle, fill=black] {};
        \node (x2) at (1,0) [draw, circle, fill=black] {};
        \node (xx) at (2,0)  {$\ldots$};        
        \node (x3) at (3,0) [draw, circle, fill=black] {};
        \node (Km) at (4.5,0) {$m$-clique};
        \node (l1) at (0,-1) [draw, circle] {};
        \node (l2) at (1,-1) [draw, circle, fill=black] {};
        \node (lx) at (2,-1)  {$\ldots$};        
        \node (l3) at (3,-1) [draw, circle, fill=black] {};
        
        \foreach \i in {1,2,3}\draw [-] (x\i)--(l\i);
        
        \node[ellipse, line width=1pt, draw = black, minimum width=4cm, minimum height=1cm] (ellipseKm) at (1.5,0) {};
        
  \end{tikzpicture}}
  }\qquad
\subfigure[The graph obtained from $K_m$ ($m\geq 2$) by joining it to $K_1=A_0\bowtie K_1$ and attaching a leaf to each vertex of $K_m$. For $m=2$ we obtain the bull graph.]{
    
    \scalebox{1}{\begin{tikzpicture}[join=bevel,inner sep=0.5mm]

        \node (u) at (1.5,1) [draw, circle] {};

        \node (x1) at (0,0) [draw, circle, fill=black] {};
        \node (x2) at (1,0) [draw, circle, fill=black] {};
        \node (xx) at (2,0)  {$\ldots$};        
        \node (x3) at (3,0) [draw, circle, fill=black] {};
        \node (Km) at (4.5,0) {$m$-clique};
        \node (l1) at (0,-1) [draw, circle, fill=black] {};
        \node (l2) at (1,-1) [draw, circle, fill=black] {};
        \node (lx) at (2,-1)  {$\ldots$};        
        \node (l3) at (3,-1) [draw, circle, fill=black] {};
        
        \foreach \i in {1,2,3}\draw [-] (u)--(x\i)--(l\i);
        
        \node[ellipse, line width=1pt, draw = black, minimum width=4cm, minimum height=1cm] (ellipseKm) at (1.5,0) {};
        
  \end{tikzpicture}}
  }\newline
\subfigure[The graph obtained from $K_m$ ($m\geq 1$) by attaching a leaf to each vertex and joining it to $A_1$ (here $m=2$).]{  
    \scalebox{1}{\begin{tikzpicture}[join=bevel,inner sep=0.5mm]

        \node (u1) at (0,2) [draw, circle, fill=black] {};
        \node (u2) at (1,2) [draw, circle, fill=black] {};
        \node (x1) at (0,0) [draw, circle, fill=black] {};
        \node (l1) at (0,-1) [draw, circle] {};
        \node (x2) at (1,0) [draw, circle, fill=black] {};
        \node (l2) at (1,-1) [draw, circle, fill=black] {};
        \node (Km) at (2.5,0) {$2$-clique};
        \draw [-] (u1)--(x1)--(x2)--(l2) (x1)--(l1) (x2)--(u1)--(x1)--(u2)--(x2);
        
        \node[rectangle, rounded corners, line width=1pt, draw = black, dashed, minimum width=2cm, minimum height=1cm] (P3) at (0.5,2.2) {$A_1$};
        \node[ellipse, line width=1pt, draw = black, minimum width=2cm, minimum height=1cm] (ellipseKm) at (0.5,0) {};
    
  \end{tikzpicture}}
  }\qquad
\subfigure[The graph obtained from $K_m$ ($m\geq 1$) by attaching a leaf to each vertex and joining it to $A_1\bowtie K_1=P_3$ (here $m=1$).]{
    \scalebox{1}{\begin{tikzpicture}[join=bevel,inner sep=0.5mm]

        \node (u1) at (0,2) [draw, circle, fill=black] {};
        \node (u2) at (1,2) [draw, circle] {};
        \node (u3) at (2,2) [draw, circle, fill=black] {};        
        \node (x1) at (1,0) [draw, circle, fill=black] {};
        \node (l1) at (1,-1) [draw, circle, fill=black] {};
        \node (Km) at (2.5,0) {$1$-clique};
        
        \draw [-] (u1)--(u2)--(u3)--(x1)--(u2) (u1)--(x1)--(l1);
        
        \node[rectangle, rounded corners, line width=1pt, draw = black, dashed, minimum width=3cm, minimum height=1.2cm] (P3) at (1,2.4) {$A_1\bowtie K_1$};
        \node[ellipse, line width=1pt, draw = black, minimum width=1.5cm, minimum height=1cm] (ellipseKm) at (1,0) {};
        
  \end{tikzpicture}}
  }\qquad
\subfigure[The graph obtained from $K_m$ ($m\geq 1$) by attaching a leaf to each vertex and joining it to $A_2=P_4$.]{
    \scalebox{1}{\begin{tikzpicture}[join=bevel,inner sep=0.5mm]

        \node (u1) at (0,2) [draw, circle, fill=black] {};
        \node (u2) at (1,2) [draw, circle, fill=black] {};
        \node (u3) at (2,2) [draw, circle, fill=black] {};        
        \node (u4) at (3,2) [draw, circle, fill=black] {};
        
        \node (x1) at (0,0) [draw, circle, fill=black] {};
        \node (x2) at (1,0) [draw, circle, fill=black] {};
        \node (xx) at (2,0)  {$\ldots$};        
        \node (x3) at (3,0) [draw, circle, fill=black] {};
        \node (Km) at (4.5,0) {$m$-clique};
        \node (l1) at (0,-1) [draw, circle] {};
        \node (l2) at (1,-1) [draw, circle, fill=black] {};
        \node (lx) at (2,-1)  {$\ldots$};        
        \node (l3) at (3,-1) [draw, circle, fill=black] {};
        
        \foreach \i in {1,2,3}\draw [-] (x\i)--(l\i) (x\i)--(u1) (x\i)--(u2) (x\i)--(u3) (x\i)--(u4);
        
        \node[ellipse, line width=1pt, draw = black, minimum width=4cm, minimum height=1cm] (ellipseKm) at (1.5,0) {};
        
        \draw [-] (u1)--(u2)--(u3)--(u4);
        
        \node[rectangle, rounded corners, line width=1pt, draw = black, dashed, minimum width=4cm, minimum height=1.2cm] (A2) at (1.5,2.4) {$A_2$};

  \end{tikzpicture}}
  }

\end{center}
  \caption{Some examples of graphs of order $n$ with total dominating identifying code number $n-1$ but a smaller identifying code. Black vertices form a minimum total dominating identifying code.}
\label{fig:extremal-examples-tID}
\end{figure}
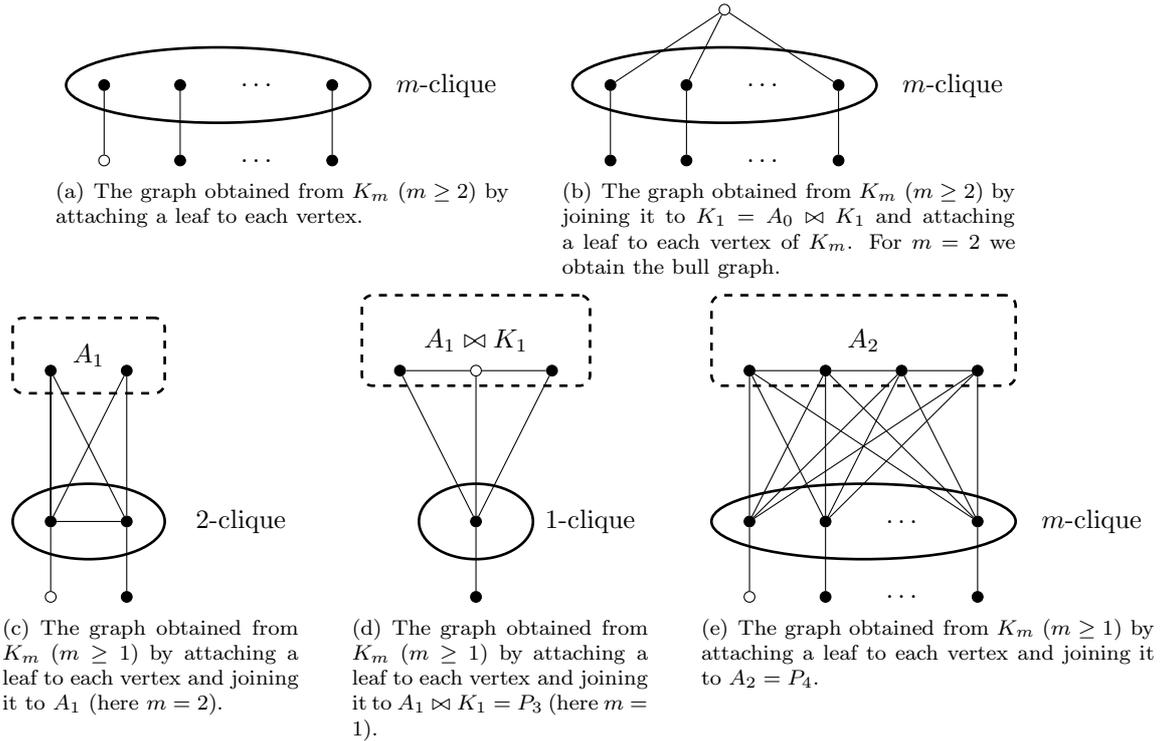

\subsection{The proof}

In the following lemma, we show that the extremal graphs are exactly the same for identification and total dominating identification, when the graphs do not contain any leaves.

\begin{lemma}\label{Lem:n-1minDegree}
If $G$ be a connected identifiable graph with minimum degree $\delta(G)\geq2$ on $n\geq 4$ vertices, then $\IDt(G)=n-1$ if and only if $\ID(G)=n-1$.
\end{lemma}
\begin{proof}
Let $G$ be a connected graph with minimum degree $\delta(G)\geq2$ on $n\geq 4$ vertices with $\IDt(G)=n-1$. Assume by contradiction that $\ID(G)\leq n-2$. We may assume that $C'$ is an identifying code of cardinality $n-2$ in $G$. We notice that $C'$ cannot be total dominating since $\IDt(G)=n-1$. Thus, there exists a vertex $v\in C'$ such that $\deg(v)=2$ (since $G$ has no degree 1 vertex) and there are two adjacent non-codewords $u$ and $u'$ and these two vertices are the only non-codewords in $G$. Since $\IDt(G)=n-1$, we cannot shift codeword $v$ to any of $u$ or $u'$ and obtain a total dominating identifying code. Since $v$ was not helpful with total domination in $C'$,  there exists a vertex $w$ such that $N[u]\cup\{u'\}=N[w]\cup\{v,u'\}$ or $N[u']\cup\{u\}=N[w]\cup\{v,u\}$. Without loss of generality we assume the first case.

Assume first that $u$ and $u'$ are not adjacent. In this case, we may shift the codeword in $w$ to $u$. Notice that the resulting code is total dominating. Moreover, since $u$ and $u'$ are not adjacent, $u$ and $w$ separate exactly the same set of vertex pairs with the exception of those with $v$ (and possibly $u'$ if $w$ is adjacent to $u'$) in them. Moreover, $I(v)$ is unique since $I(v)=\{u,v\}$, and $u'$ is the only vertex with $v$ in its $I$-set while not having $u$ in its $I$-set. Hence, the resulting code is total dominating identifying. Moreover, it is a total dominating identifying code with cardinality of $n-2$, a contradiction. Hence, we may assume from now on that $u$ and $u'$ are adjacent. If we again do the same shift of codewords, then we notice that we have a total dominating identifying code unless $w$ was the vertex which separated $u$ and $u'$. That is, $N[u]=N[u']\cup\{w\}$.

Recall that we could not shift the codeword in $v$ to $u'$ and get an identifying code. Since $u'$ separates $w$ and $u$, we have a vertex $w'\neq w$ with $N[u]=N[w']\cup \{v\}$ or $N[u']=N[w']\cup \{v\}$. If $N[u]=N[w']\cup \{v\}$, then $N[w]\cup\{v,u'\}=N[w']\cup\{v\}$. Hence, $u'$ is the only vertex which can separate $w$ and $w'$ and $u'$ has to be a codeword in any identifying code, a contradiction since $C'$ did not contain it. Hence, we may assume that $N[u']=N[w']\cup \{v\}$. However, now we may shift the codeword from $w'$ to $u'$ and get a total dominating identifying code. Hence, we have $\IDt(G)=n-2$, a contradiction, and we have $\ID(G)=n-1$.

The other direction is clear. If $\ID(G)=n-1$, then $\IDt(G)=n-1$ since $G$ is not $P_3$.\end{proof}

To exactly characterize the extremal graphs for total dominating identification, we require some lemmas which will later be utilized in the induction.

\begin{lemma}\label{Lem:deletion}
Let $G$ be a connected graph of order $n\geq 5$ other than a star, with a leaf $u$ and an adjacent support vertex $v$. If $\IDt(G)= n-1$, then $\IDt(G-u-v)\geq n-3$ and $G-u-v$ is identifiable and connected.
\end{lemma}
\begin{proof}
Let $G$ be a connected graph other than a star with $\IDt(G)=n-1\geq4$ with leaf $u$ and adjacent support vertex $v$. We denote graph $G-u-v$ by $G_v$. We prove the following facts.

\medskip

\noindent\textbf{(1) $G_v$ has no components of size~2.} 
Suppose on the contrary that such a component exists in $G_v$, say, with $x,y$ as its vertices and $x\in N_G(v)$ (thus $y\notin N_G(v)$ since $G$ is identifiable). Now, $V(G)\setminus\{y,u\}$ is a total dominating identifying code of cardinality $n-2$ in $G$, a contradiction. Indeed, the code is clearly total dominating and $y$ is the only vertex with $I(y)=\{x\}$, $x$ is the only one with $I(x)=\{v,x\}$ since $n\geq5$, $u$ the only one with $I(u)=\{v\}$, $v$ is the only other vertex which is adjacent to $x$ and hence is separated from the rest of codewords. Finally, the other vertices will have unique $I$-sets since $G$ is identifiable.

\medskip

\noindent\textbf{(2) $G_v$ is connected}. Suppose on the contrary that we have several components in $G_v$. By the above paragraph, none of them has size~2.

Consider the case where each component in $G_v$ has at least three vertices. Notice that if a component, say $A$, is not identifiable, then $G[A\cup\{v\}]$ is identifiable (and not isomorphic to $P_3$), by Proposition~\ref{Prop:IDtn-1} there is a total dominating identifying code of $G[A\cup\{v\}]$ of cardinality at most $|A|$, and it must contain vertex $v$. On the other hand, if a component of $G_v$ is identifiable, then again by Proposition~\ref{Prop:IDtn-1} it has a total dominating identifying code which does not contain each vertex in that component (unless the component is $P_3$, in which case that component together with $v$ has a total dominating identifying code of size~3, with a non-codeword other than $v$). If we now consider graph $G$, then, by combining the codes in each component (together with $u$ and $v$), we find a code which contains at most $n-2$ codewords. Resulting code is clearly total dominating and each vertex within the  components is separated by codewords within those components or by $v$. Moreover, $v$ and $u$ are separated from every other vertex by $u$. Hence, we have a total identifying code unless $I(u)=I(v)$. However, if $I(u)=I(v)$ holds, then we can just move the codeword from $u$ to any other vertex adjacent to $v$. Now $u$ is the only vertex adjacent to only $v$ and each other vertex adjacent to $v$ already had, before the codeword shift, another codeword which is adjacent to it. This leads to a contradiction.

Finally, to show that $G$ is connected, it remains to deal with the case where some component of $G_v$ is a single vertex, that is, $v$ has at least two adjacent leaves. By (1) we know that no component of $G_v$ has size~2. Let $U=\{u,u_1,\ldots,u_k\}$, 
for some $k\geq 1$, be the set of leaves adjacent to $v$ in $G$. Denote graph $G-U-v$ by $G''$. Similarly as above, we can check that $G''$ is not $P_2$ or $P_3$. Assume first that $G''$ is not identifiable. Then, $G-U$ is identifiable, has size at least~3 (because $G$ is not a star) and is not isomorphic to $P_3$ and hence, by Proposition~\ref{Prop:IDtn-1}, it has a total dominating identifying code $C'$ of size at most $n-|U|-1$. Moreover, $v\in C'$ since it is the only vertex which separates some pair of closed twins in $G''$. However, now $C'$ together with all but one vertex in $U$ is a total dominating identifying code of $G$ of cardinality at most $n-2$, a contradiction. Hence, we may assume that $G''$ is identifiable. Moreover, if $G''$ has a total dominating identifying code $C'$ of size at most $n-|U|-3$, then $C'\cup U\cup\{v\}$ is a total dominating identifying code of size at most $n-2$ in $G$, a contradiction. Thus, $\IDt(G'')= n-|U|-2$ since $G''$ is not isomorphic to $P_3$. 

Let $C'$ be a total dominating identifying code of $G''$ with $|C'|=n-|U|-2$. Assume first that $C'$ contains a neighbour of $v$ (this is true in particular if $\deg(v)\geq |U|+2$ in $G$). Then $C'\cup\{v\}$ together with all the vertices of $U$ but one is a total dominating identifying code of size $n-2$ in $G$ since $v$ has a codeword neighbour in $C'$, again a contradiction. Thus, we may assume that $\deg(v)=|U|+1$ in $G$ (we denote by $w$ the neighbour of $v$ not in $U$), and $w\not\in C'$. Thus, $w$ is not a support vertex. Then, $G''-w$ is identifiable. If $G''-w$ is $P_3$, then it is easy to check that $G$ has a total dominating identifying code of size at most $n-2$. Otherwise, if $G''-w$ is connected, then, by Proposition~\ref{Prop:IDtn-1}, it satisfies $\IDt(G''-w)\leq n-|U|-3$. If $G''-w$ is disconnected, then none of the components is a $P_2$ (otherwise $w$ would be a codeword in $C'$) and if a  component is a $P_3$, then each vertex in that $P_3$ is a codeword in $C'$. We can now just shift one of the codewords in the $P_3$ to $w$ in $C'$ and obtain a total dominating identifying code of $G''$. (The codeword that can be shifted depends on which edges exist between $w$ and the $P_3$-component. At the beginning of this paragraph we have shown that if $v$ has an adjacent codeword vertex in a total dominating identifying code of $G''$ of size at most $n-|U|-2$, then we have a total dominating identifying code of size $n-2$ in $G$, a contradiction. Hence, we may assume that we do not have any such $P_3$-components in $G''-w$. Thus, each component in $G''-w$ has at least four vertices and hence, by Proposition~\ref{Prop:IDtn-1}, $\IDt(G''-w)\leq n-|U|-3$. Now, in $G$, the code $C''\cup\{v,w\}$ together with all the vertices of $U$ but one, is total dominating identifying with cardinality at most $n-2$, a contradiction.


Hence we have proved that $G_v$ is connected.

\medskip

\noindent\textbf{(3) $G_v$ is identifiable.} By contradiction, assume $G_v$ has some closed twins. Assume first that $G_v$ has three mutually twin vertices $x,y$ and $z$ such that $N_{G_v}[x]=N_{G_v}[y]=N_{G_v}[z]$. Now, $v$ cannot separate all three of these vertices in $G$ and hence, we have a contradiction. 

Assume next that we have at least two disjoint pairs of closed twins, that is, $N_{G_v}[x]=N_{G_v}[y]$ and $N_{G_v}[z]=N_{G_v}[w]$. We may assume that $v\in N_G(x)$ and $v\in N_G(z)$ but $v\not\in N_G(y)$ and $v\not\in N_G(w)$. Now, $V(G)\setminus \{y,w\}$ is a total dominating identifying code in $G$. Clearly the code is total dominating. Moreover, $v$ separates $x$ and $y$ as well as $z$ and $w$. Furthermore, $u$ separates $v$ from other vertices and $x$ separates $v$ from $u$. Since $N_{G_v}[x]=N_{G_v}[y]$ and $N_{G_v}[z]=N_{G_v}[w]$, adding $w$ or $y$ to the code will not separate any new vertices and since $G$ is identifiable, this code is an identifying code.

Thus, we now assume that there is exactly one pair of closed twins in $G_v$, that is, $N_{G_v}[x]=N_{G_v}[y]$ where $v\in N_G[x]\setminus N_G[y]$. Notice that by (1), $G_v$ has no components of size~2, hence $x$ and $y$ have a common neighbour, $z$. Then, consider the graph $G''=G-\{u,v,x\}$. Notice that $G_v'$ is identifiable since $G$ is identifiable and if $N_{G''}[a]=N_{G''}[b]$ for some vertices $a,b$, then $u$ cannot separate them in $G$, $v$ cannot be the only one to separate them in $G$ since there is exactly one pair of closed twins in $G_v$, and $x$ cannot separate them either since $N_{G_v}[x]=N_{G_v}[y]$. Thus, $G''$ is identifiable, as claimed.

Notice that if $G''$ is isomorphic to $P_3$, then $G$ is one of four possible graphs (see Figure~\ref{fig:smallgraphs}), and in each case one can check that $\IDt(G)\leq 4=n-2$, a contradiction.

\begin{figure}[ht]
\begin{center}

  \subfigure[]{
    \scalebox{1}{\begin{tikzpicture}[join=bevel,inner sep=0.5mm]
        \node (u) at (0,0) [draw, circle, label={90:$u$}] {};
        \node (v) at (1,0) [draw, circle, label={90:$v$}, fill=black] {};
        \node (x) at (2,0) [draw, circle, label={90:$x$}, fill=black] {};
        \node (y) at (3,0) [draw, circle, label={90:$y$}] {};
        \node (z) at (2.5,-1) [draw, circle, label={-90:$z$}, fill=black] {};
        \node (t) at (1.5,-1) [draw, circle, fill=black] {};
        \draw [-] (u)--(v)--(x)--(y)--(z)--(x) (z)--(t);
  \end{tikzpicture}}
  }\qquad
      \subfigure[]{
    \scalebox{1}{\begin{tikzpicture}[join=bevel,inner sep=0.5mm]
        \node (u) at (0,0) [draw, circle, label={90:$u$}] {};
        \node (v) at (1,0) [draw, circle, label={90:$v$}, fill=black] {};
        \node (x) at (2,0) [draw, circle, label={90:$x$}, fill=black] {};
        \node (y) at (3,0) [draw, circle, label={90:$y$}, fill=black] {};
        \node (z) at (2.5,-1) [draw, circle, label={-90:$z$}] {};
        \node (t) at (1.5,-1) [draw, circle, fill=black] {};
        \draw [-] (u)--(v)--(x)--(y)--(z)--(x) (v)--(t)--(z);
  \end{tikzpicture}}
  }\qquad
    \subfigure[]{
    \scalebox{1}{\begin{tikzpicture}[join=bevel,inner sep=0.5mm]
        \node (u) at (0,0) [draw, circle, label={90:$u$}] {};
        \node (v) at (1,0) [draw, circle, label={90:$v$}, fill=black] {};
        \node (x) at (2,0) [draw, circle, label={90:$x$}] {};
        \node (y) at (3,0) [draw, circle, label={90:$y$}, fill=black] {};
        \node (z) at (2.5,-1) [draw, circle, label={-90:$z$}, fill=black] {};
        \node (t) at (1.5,-1) [draw, circle, fill=black] {};
        \draw [-] (u)--(v)--(x)--(y)--(z)--(x) (v)--(z)--(t);
  \end{tikzpicture}}
  }\qquad
    \subfigure[]{
    \scalebox{1}{\begin{tikzpicture}[join=bevel,inner sep=0.5mm]
        \node (u) at (0,0) [draw, circle, label={90:$u$}] {};
        \node (v) at (1,0) [draw, circle, label={90:$v$}, fill=black] {};
        \node (x) at (2,0) [draw, circle, label={90:$x$}, fill=black] {};
        \node (y) at (3,0) [draw, circle, label={90:$y$}, fill=black] {};
        \node (z) at (2.5,-1) [draw, circle, label={-90:$z$}] {};
        \node (t) at (1.5,-1) [draw, circle, fill=black] {};
        \draw [-] (u)--(v)--(x)--(y)--(z)--(x) (v)--(z)--(t)--(v);
  \end{tikzpicture}}
  }

\end{center}
  \caption{The possibilities for graph $G$ when $G''$ is isomorphic to $P_3$, in part (3) of the proof of Lemma~\ref{Lem:deletion}. The black vertices form total dominating sets of size $n-2$.}
\label{fig:smallgraphs}
\end{figure}
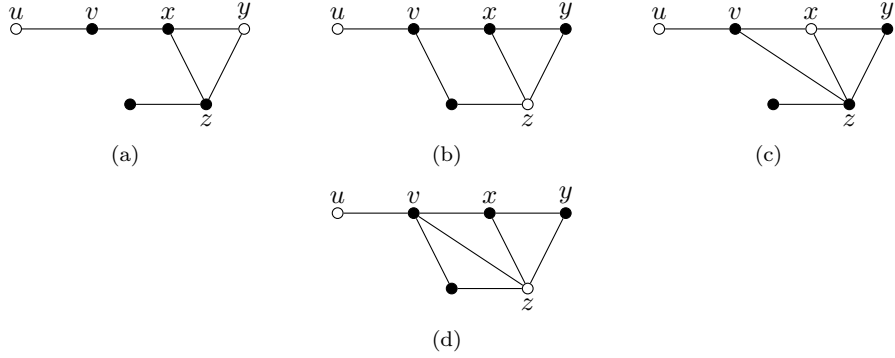

Hence, by Proposition~\ref{Prop:IDtn-1}, we have $\IDt(G'')\leq n-4$. Let $C''$ be an optimal total dominating identifying code of $G''$.

Observe first that $|C''|=n-4$. Indeed, if $|C''|\leq n-5$, then $C=C''\cup\{u,v,x\}$ is a total dominating identifying code of $G$ of cardinality at most $n-2$. Indeed, $C$ is clearly total dominating and $v$ is separated from all vertices except $u$ by $u$ and $u$ is separated from $v$ by $x$. Moreover, $u$ is the only vertex with $I(u)=\{u,v\}$. Furthermore, if $I(x)=I(a)$ for some vertex $a$, then $I(a)=I(y)\cup \{v\}$ and $a$ and $y$ are not separated in $G''$, a contradiction. Thus, $C$ is total dominating identifying with $n-2$ codewords since all other vertices are dominated and separated by $C'$. Hence $|C''|=n-4$.

Consider the two codes $C_x=C''\cup\{v,x\}$ and $C_y=C''\cup\{v,y\}$. $C_x$ is a total dominating set, and $C_y$ is also a total dominating set, except if $v$ has degree~2 in $G$. Observe that for both codes, $I(u)=\{v\}$ and is unique if $\deg(v)\geq3$. All vertex pairs in $G''$ are separated by the vertices in $C''$. Moreover, in both codes, if some vertex $b$ of $G''$ is not separated from $x$, then this means that $I(b)=I(y)\cup\{v\}$ but then $b$ and $y$ are not separated by $C''$ in $G''$, a contradiction. Thus, in both codes, $x$ is separated from all vertices of $G$, except possibly $v$. Hence, for each of the two codes, if $v$ is also separated from all other vertices and has a neighbour in the code, then that code is a total dominating identifying code of size at most $n-2$, a contradiction, and we are done.

Hence, we assume that neither $C_x$ nor $C_y$ are total dominating identifying codes. Since $C_x$ is total dominating, it is not identifying; hence, by the above discussion, there is some vertex $b$ with $I(v)=I(b)$ in $C_x$. Thus, $b$ is dominated by $v$ and $x$ (possibly, $b=x$): $b$ must be a neighbour of $y$. Then, $y\notin C''$, for otherwise, $b$ and $v$ would be separated by $y$ in $C_x$, a contradiction. Thus, we have $N[v]\triangle N[b]=\{u,y\}$. If $b=x$, then $C_y$ is a total dominating identifying code, indeed $x,y$ have at least one common neighbour in $C''$, which is a neighbour of $v$, so $C_y$ is total dominating. Moreover, all neighbours of $v$ except $u$ are neighbours of $y$, so $v$ is also separated from all other vertices either by $v$ or by $y$. Therefore, we have $b\neq x$. Hence, $\deg_G(v)\geq 3$ and $C_y$ is total dominating. Thus, $C_y$ is not identifying, that is, there is a vertex $c$ with $I(c)=I(v)$ in $C_y$. Hence, $c$ is not adjacent to $y$ (hence, not to $x$) and is adjacent to $b$, and $N[c]\triangle N[v]=\{u,x\}$. It follows that $N[b]\triangle N[c]=\{x,y\}$, however that is a contradiction, since $N_G[x]=N_G[y]\cup \{v\} $.


Thus, we have shown that $G_v$ is identifiable.

\medskip

\noindent\textbf{(4) \boldmath{$\IDt(G_v)\geq n-3$}.} Suppose on the contrary that there exists a total dominating identifying code $C'$ with cardinality $n-4$ in $G_v$. Consider code $C=C'\cup\{u,v\}$ in $G$. It is clearly total dominating and has cardinality of $n-2$. Moreover, $u$ separates itself and $v$ from all other vertices. Hence, we are done unless $I(v)=I(u)=\{v,u\}$, thus, assume that $N(v)\cap C'=\emptyset$. Since $G$ is connected, $v$ has at least one neighbour, $w$, other than $u$. Let us instead consider code $C_w=C\cup\{w,v\}$. Again, the code is clearly total dominating. Moreover, $u$ is the only vertex with $I(u)=\{v\}$ since $C'$ is total dominating in $G_v$. Furthermore, $I(v)=\{v,w\}$ and if $I(a)=\{v,w\}$ for some vertex $a$, then $a$ is not dominated by $C'$ in $G_v$, a contradiction. Thus, we found a total dominating identifying code of cardinality $n-2$ in $G$, a contradiction.
\end{proof}

In the following lemma, we find the set of graphs of order $n$ which have (usual) identifying code number $n-1$ and to which we may add a leaf and a support vertex so that the resulting graph has total dominating identifying code number $n'-1=n+1$, but larger than the usual identifying code number. Small stars are  special cases for the lemma that are  excluded. In particular, star $K_{1,3}$ can actually be constructed from $A_1\in \mathcal{A}$ by adding a universal vertex $v$ and a leaf $u$ to $v$. The star $K_{1,2}$ is isomorphic to $P_3$ and the bull graph (illustrated in Figure~\ref{fig:smallgraphs-P3}(d)) can be constructed from it by adding a non-universal support vertex. However, the bull graph can also be constructed from $K_1$ by joining it to a copy of $K_2$ and adding leaves to the two newly added vertices. 



\begin{lemma}\label{Lem:ExtremalIDt from ID}
Let $G$ be a connected graph on $n\geq3$ vertices with support vertex $v$ and an adjacent leaf $u$ and $G'=G-u-v\neq K_{1,p}$ with $p\leq 3$. If $\IDt(G)=n-1$ and $\ID(G')=n-3$, then $G'\in \mathcal{A}^*\cup(\mathcal{A}^*\bowtie K_1)$and $v$ is a universal vertex in $G$.
\end{lemma}
\begin{proof}
Since $|V(G')|=n-2$ and $\ID(G')=n-3$, we have $G'\in \{K_{1,t}\mid t\geq2\}\cup \mathcal{A}^*\cup(\mathcal{A}^*\bowtie K_1)$ by Theorem~\ref{The:IDn-1}. Based on this, we distinguish several cases. Notice that we have $n\geq5$. 

\medskip

\noindent\textbf{Case 1: $G'$ is $A_k\in\mathcal{A}$.} Assume by contradiction that $v$ is not a universal vertex in $G$. By Lemma~\ref{Lem:deletion}, $G'$ must be connected, hence we have $k\geq 2$. Moreover, assume that $i$ is the smallest integer for which $x_i\not\in N(v)$ and $x_{i-1}\in N(v)$. If $x_1\not\in N(v)$, then let $i$ be the smallest integer for which $x_i\not\in N(v)$ and $x_{i+1}\in N(v)$. Moreover,  assume that $v$ does not separate the two maximum cliques in $A_k$, that is, we do not have $N(v)\cap V(A_k)=\{x_1,\dots,x_k\}$ or $N(v)\cap V(A_k)=\{x_{k+1},\dots,x_{2k}\}$. Assume now that $i\leq k$ and $x_{i-1}\in N(v)$. Thus, $v$ separates $x_i$ and $x_{i-1}$. Consider code $C=\{v,u\}\cup V(A_k)\setminus\{x_{2k},x_{i+k-1}\}$. Code $C$ is clearly a total dominating set and it has $n-2$ vertices. Moreover, $V(A_k)\setminus \{x_{2k}\}$ is an identifying code in $A_k$ and codeword $x_{i+k-1}$ is used to separate vertices $x_i$ and $x_{i-1}$ from each other. Furthermore, each vertex $x_j$, $2\leq j\leq 2k$, $j\neq i,i-1$, is identified in the same way as in $A_k$, and vertices $x_i$ and $x_{i-i}$ are separated by $v$. Finally, since $v$ and $u$ are the only vertices with $u$ in their $I$-sets, they have unique $I$-sets. The case where $x_1\not\in N(v)$ is similar with the exception that we have $x_1$ as the non-codeword instead of $x_{2k}$. 
Now, we are left with the case where $N(v)\cap V(A_k)=\{x_1,\dots,x_k\}$ or $N(v)\cap V(A_k)=\{x_{k+1},\dots,x_{2k}\}$. These two cases are symmetric, thus without loss of generality, we may assume that the first one holds. Consider the code $C=\{v,u\}\cup V(A_k)\setminus\{x_1,x_{2k}\}$. Recall that $V(A_k)\setminus \{x_1\}$ is an identifying code in $A_k$. Moreover, the only identical $I$-sets with the code $V(A_k)\setminus \{x_1,x_{2k}\}$ in $A_k$ are $I(x_k)$ and $I(x_{k+1})$. However, in $G$, the codeword $v$ separates these two vertices. Thus, $C$ is a total dominating identifying code in $G$. 

Thus, $v$ is a universal vertex, as claimed.

\medskip

\noindent\textbf{Case 2: $G'\in \mathcal{A}$ but $G'$ is not any graph $A_i$.} Assume that $v$ is not a universal vertex in $G$. Recall that $G'$ is constructed with a sequence of joins of graphs $A_{i_j}$. Notice that if there exists a subgraph $A_{i_j}$ of $G'$ such that $v$ is adjacent to some but not all of the vertices of that subgraph, then we can find a new non-codeword as in Case~1 if $i_j\geq2$. When $i_j=1$, $\deg_G(v)\geq3$ and $v$ separates vertices in $A_{i_j}$, we can proceed as in Case 1, that is, have both vertices of $A_{i_j}$ as non-codewords. When $i_j=1$ and $N(v)=\{u,x_1\}$ where $x_1\in V(A_{i_j})$, we can consider total dominating identifying code $ V(G)\setminus \{u,x_2\}$ where $x_2\neq x_1$ is the other vertex of $A_{i_j}$.

Moreover, if, for each $j$, every vertex in subgraph $A_{i_j}$ is either adjacent or non-adjacent to $v$ and, say, $V(A_{i_1})\subseteq N(v)$ and $V(A_{i_2})\cap N(v)=\emptyset$, then we may choose as the two non-codewords the vertices corresponding to $x_1$ in each of the subgraphs $A_{i_1}$ and $A_{i_2}$. Without $v$, we could not do this since nothing would separate vertices $x_{i_1+1}$ and $x_{i_2+1}$ in the corresponding subgraphs, but now $v$ separates them. Hence, we can construct a total dominating identifying code of size at most $n-2$, a contradiction, and $v$ is universal.

\medskip

\noindent\textbf{Case 3: $G'\in \mathcal{A}\bowtie K_1$.} Notice that $A_1\bowtie K_1$ is $K_{1,2}$ and hence, by our assumptions, we do not have to consider it. Hence, $|V(G')|\geq5$. Denote by $y$ the universal vertex of $G'$. Recall (see Proposition~\ref{prop:extremal-family-IDtn-1}) that the only minimum identifying code in $G'$ consists of every vertex except $y$. Assume first that there does not exist any vertex $z\in V(G')$ with $N[v]=N[z]\cup\{u\}$. Then we may consider code $C=V(G)\setminus\{u,y\}$. Code $C$ is total dominating since $G'\neq A_1\bowtie K_1$ and $|V(G')|\geq5$. Moreover, all vertices in $G'$ have pairwise distinct $I$-sets since $C\setminus \{v\}$ is a total dominating identifying code in $G'$. Furthermore, $u$ is the only vertex with $I(u)=\{v\}$ while $v$ is separated from other vertices since it has a unique closed neighbourhood. 

Assume then that there exists a vertex $z\in V(G')$ with $N[v]=N[z]\cup\{u\}$ and $z\neq y$. Now, we consider code $C=V(G)\setminus \{z,y\}$. Again, code $C$ is a total dominating set since $V(G')\setminus \{y\}$ is total dominating in $G'$ and $v$ is adjacent to any vertex which would be dominated by $z$. Moreover, $V(G')\setminus \{y\}$ is an identifying code in graph $G'$. Codeword $u$ separates $u$ and $v$ from other vertices in $G$ and $|I(v)|>|I(u)|$. Since $N[v]=N[z]\cup\{u\}$, any vertices that would be separated by $z$ in $G'$ by code $V(G')\setminus \{y\}$ are now separated by $v$. Hence, $C$ is a total dominating identifying code in $G$ of cardinality $n-2$ and $v$ is universal.


\medskip

\noindent\textbf{Case 4: $G'\in \{K_{1,t}\mid t\geq2\}$}. 
We first show that $G'$ has exactly two leaves. Consider on the contrary that $G'=K_{1,t}$. By our assumption that $G'$ is not $K_{1,p}$ for $p\leq 3$, we have $t\geq 4$. Denote by $w$ the central vertex of $G'$ and by $\{w_1,\dots,w_t\}=L(G')$ the $t$ leaves of $G'$. Observe first that if $v$ is adjacent to at most $t-2$ leaves of $G'$, then the graph $G-w-w_1$, where $w_1\not\in N_G(v)$, is disconnected and hence, by Lemma~\ref{Lem:deletion}, $\IDt(G)<n-1$, a contradiction. Then, consider the case where $v$ is adjacent to $t-1$ leaves $w_2,\dots, w_t$ (possibly, $v$ is adjacent to $w$ as well). We choose $C=V(G)\setminus\{w_2,u\}$, and show it is a total dominating identifying code of cardinality $n-2$. Observe that it is clearly total dominating. Moreover, $C$ is identifying since $w_1$ separates $w_1$ and $w$ from other vertices and $|I(w)|>|I(w_1)|$, $u$ is the only vertex with $I(u)=\{v\}$, $v$ is adjacent to multiple codewords in $L(G')$ and hence separated from the leaves in $L(G')$ and each codeword in $L(G')$ is separated by itself from all other leaves. Thus, $C$ is identifying, a contradiction.

Assume then that $v$ is adjacent to each leaf in $G'$ (possibly, $v$ is adjacent to $w$). Now we choose $C=V(G)\setminus\{w,w_1\}$. Again, code $C$ is clearly total dominating and has cardinality of $n-2$. Moreover, it is identifying. Indeed, $u$ and $v$ are separated from other vertices by $u$ and $|I(v)|>|I(u)|$, $w$ is clearly separated from the leaves of $G'$, $I(w_1)=\{v\}$ and is unique, and each leaf codeword is separated from the other leaves by itself. Hence, the claim follows.
%
\end{proof}

Now we are ready to prove the exact characterization of extremal graphs from Theorem~\ref{The:IDtn-1Char}.

\begin{proof}[Proof of Theorem~\ref{The:IDtn-1Char}]
By Proposition~\ref{Prop:IDtn-1}, $\IDt(G)= n$ if and only if $G=P_3$.

Let us first see that the graphs of the statement are indeed extremal. If $G\in  \{K_{1,t}\mid t\geq2\}\cup \mathcal{A}^*\cup(\mathcal{A}^*\bowtie K_1)$, then $\IDt(G)\geq n-1$, since $\ID(G)\geq n-1$ by Theorem~\ref{The:IDn-1}. If $G'=G''\bowtie K_m$ where $m\geq1$ and $G''\in \mathcal{A}\cup(\mathcal{A}\bowtie K_1)$, and we add attach a leaf to each vertex in the clique $K_m$, then $\IDt(G)\geq n-1$ by Proposition~\ref{prop:tID-new-extremal}.

\medskip

We then show that these are the only graphs attaining the extremal value of $n-1$. Let $\IDt(G)=n-1$ and $G$ be a graph other than a star. By Lemma~\ref{Lem:n-1minDegree}, if there are no leaves in $G$, then $\ID(G)=n-1$ and we are done by Theorem~\ref{The:IDn-1}. Thus, we assume that $G$ has at least one leaf $u$ and an adjacent support vertex $v$, and we proceed by induction on the number $n$ of vertices. 

For the base cases, let us first go through all the graphs with $3\leq n\leq 6$, with leaves, and $\IDt(G)=n-1$. Let $v$ be a support vertex in $G$ and $u$ be the adjacent leaf. The only identifiable graph with $n=3$ is $P_3$, which is isomorphic to $K_{1,2}$ and hence in the family. When $n=4$ we have $P_4$, which is isomorphic to $A_2$ and in the family as well. 

For $n=5$, due to Lemma~\ref{Lem:deletion} we are only interested in the graphs for which $G-u-v$ is identifiable, connected and $\IDt(G-u-v)\geq2$; that is, $G-u-v$ is $P_3$. The possible graphs are depicted in Figure~\ref{fig:smallgraphs-P3}. The three graphs in (a), (b) and (c) have a total dominating identifying code of size at most $n-2$, while the two other ones are in the extremal family. Indeed, (d) is the bull graph, which is obtained from $K_1$ (i.e. $A_0\bowtie K_1$) in $\mathcal A\bowtie K_1$ by joining it to $K_2$ and adding an adjacent leaf to each vertex of $K_2$. Moreover, (e) is obtained from $P_3$ (i.e. $A_1\bowtie K_1$) in $\mathcal A\bowtie K_1$ by joining it to $K_1$ and adding an adjacent leaf to its vertex. 

When $n=6$, again by Lemma~\ref{Lem:deletion}, we are only interested in the connected identifiable graphs for which $\IDt(G-u-v)\geq3$, that is, for which $G-u-v$ is $P_4$, $C_4$ or $K_{1,3}$. 
There exist nine such graphs for which $G-u-v$ is $P_4$, five graphs for which $G-u-v$ is $C_4$ and seven graphs for which $G-u-v$ is $K_{1,3}$. However, by Lemma~\ref{Lem:deletion}, we may omit each graph $G$ from which we may obtain an unconnected or non-identifiable graph by deleting a leaf-support vertex pair. After that we are left with six  graphs for which $G-u-v$ is $P_4$, five graphs for which $G-u-v$ is $C_4$ and four graphs for which $G-u-v$ is $K_{1,3}$, see Figure~\ref{fig:smallgraphs-P4-C4}. Apart from graphs (d), (f), (k) and (n), all have a total dominating identifying code of size at most $n-2$. Graph (d) is in the family, since it is isomorphic to the empty graph $A_0$ in $\mathcal A$, to which has been joined a copy of $K_3$ with a leaf attached to each vertex. Graphs (f) and (k) are in the family as well, as they are either $P_4$ (i.e. $A_2$ in $\mathcal A$) or $C_4$ (i.e. $A_1\bowtie A_1$ in $\mathcal A$) joined to $K_1$ whose vertex a leaf is attached to. Finally, (n) is also in the family, as it is $A_1$ joined to $K_2$ whose vertices we have attached leaves.

Hence, we can assume from now on that $n\geq7$ and we proceed with the inductive step.

\begin{figure}[ht]
\begin{center}

  \subfigure[]{
    \scalebox{1}{\begin{tikzpicture}[join=bevel,inner sep=0.5mm]
        \node (u) at (0,1) [draw, circle, label={90:$u$}] {};
        \node (v) at (1,1) [draw, circle, label={90:$v$}, fill=black] {};
        \node (x) at (0,0) [draw, circle, fill=black] {};
        \node (y) at (1,0) [draw, circle, fill=black] {};
        \node (z) at (2,0) [draw, circle] {};
        \draw [-] (u)--(v)--(x)--(y)--(z);
  \end{tikzpicture}}
  }\qquad
  \subfigure[]{
    \scalebox{1}{\begin{tikzpicture}[join=bevel,inner sep=0.5mm]
        \node (u) at (0,1) [draw, circle, label={90:$u$}] {};
        \node (v) at (1,1) [draw, circle, label={90:$v$}, fill=black] {};
        \node (x) at (0,0) [draw, circle, fill=black] {};
        \node (y) at (1,0) [draw, circle, fill=black] {};
        \node (z) at (2,0) [draw, circle] {};
        \draw [-] (u)--(v)--(y) (x)--(y)--(z);
  \end{tikzpicture}}
  }\qquad
    \subfigure[]{
    \scalebox{1}{\begin{tikzpicture}[join=bevel,inner sep=0.5mm]
        \node (u) at (0,1) [draw, circle, label={90:$u$}] {};
        \node (v) at (1,1) [draw, circle, label={90:$v$}, fill=black] {};
        \node (x) at (0,0) [draw, circle, fill=black] {};
        \node (y) at (1,0) [draw, circle] {};
        \node (z) at (2,0) [draw, circle, fill=black] {};
        \draw [-] (u)--(v)--(x)--(y)--(z)--(v);
  \end{tikzpicture}}
  }\qquad
    \subfigure[]{
    \scalebox{1}{\begin{tikzpicture}[join=bevel,inner sep=0.5mm]
        \node (u) at (0,1) [draw, circle, label={90:$u$}, fill=black] {};
        \node (v) at (1,1) [draw, circle, label={90:$v$}, fill=black] {};
        \node (x) at (0,0) [draw, circle] {};
        \node (y) at (1,0) [draw, circle, fill=black] {};
        \node (z) at (2,0) [draw, circle, fill=black] {};
        \draw [-] (u)--(v)--(x)--(y)--(z) (v)--(y);
  \end{tikzpicture}}
  }\qquad
    \subfigure[]{
    \scalebox{1}{\begin{tikzpicture}[join=bevel,inner sep=0.5mm]
        \node (u) at (0,1) [draw, circle, label={90:$u$}, fill=black] {};
        \node (v) at (1,1) [draw, circle, label={90:$v$}, fill=black] {};
        \node (x) at (0,0) [draw, circle, fill=black] {};
        \node (y) at (1,0) [draw, circle] {};
        \node (z) at (2,0) [draw, circle, fill=black] {};
        \draw [-] (u)--(v)--(x)--(y)--(z)--(v)--(y);
  \end{tikzpicture}}
  }

\end{center}
  \caption{The possibilities for graph $G$ when $G-u-v$ is isomorphic to $P_3$, in the proof of Theorem~\ref{The:IDtn-1Char}. The black vertices form total dominating identifying codes.}
\label{fig:smallgraphs-P3}
\end{figure}
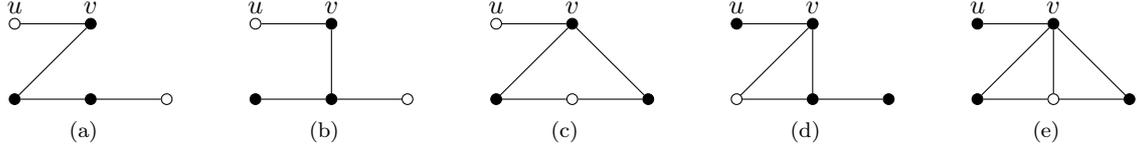

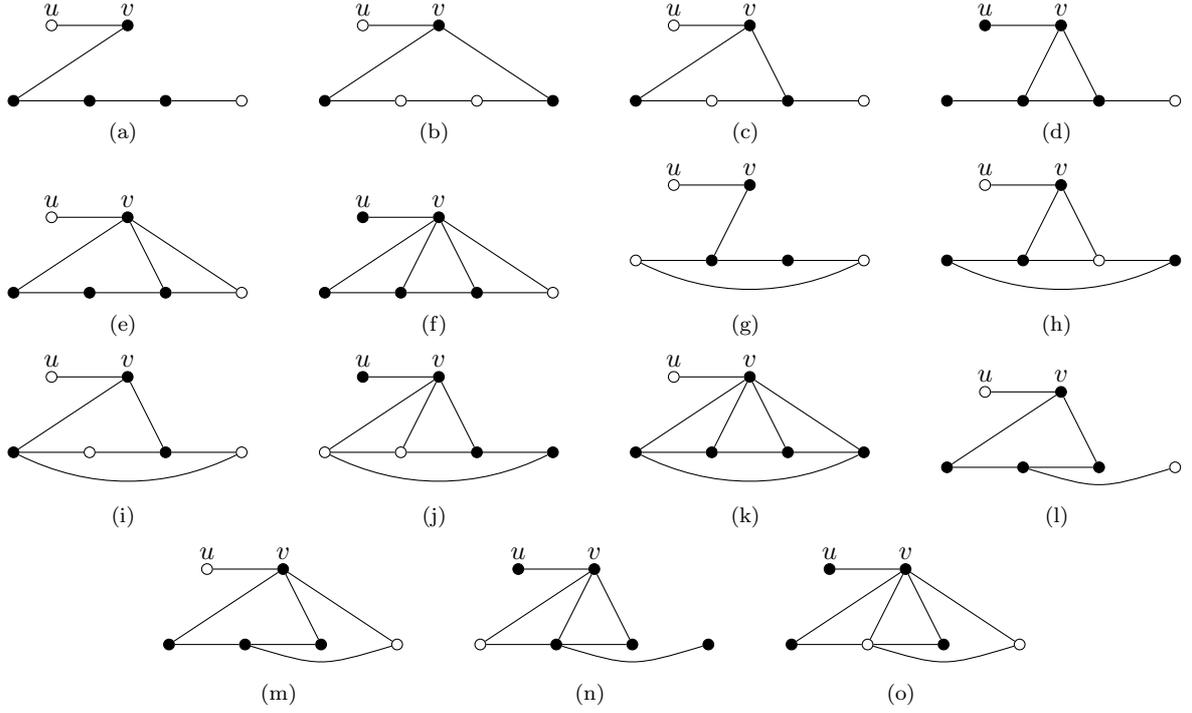
\begin{figure}[ht]
\begin{center}

  \subfigure[]{
    \scalebox{1}{\begin{tikzpicture}[join=bevel,inner sep=0.5mm]
        \node (u) at (0,1) [draw, circle, label={90:$u$}] {};
        \node (v) at (1,1) [draw, circle, label={90:$v$}, fill=black] {};
        \node (x) at (-0.5,0) [draw, circle, fill=black] {};
        \node (y) at (0.5,0) [draw, circle, fill=black] {};
        \node (z) at (1.5,0) [draw, circle, fill=black] {};
        \node (t) at (2.5,0) [draw, circle] {};        
        \draw [-] (u)--(v)--(x)--(y)--(z)--(t);
  \end{tikzpicture}}
  }\qquad
  \subfigure[]{
    \scalebox{1}{\begin{tikzpicture}[join=bevel,inner sep=0.5mm]
        \node (u) at (0,1) [draw, circle, label={90:$u$}] {};
        \node (v) at (1,1) [draw, circle, label={90:$v$}, fill=black] {};
        \node (x) at (-0.5,0) [draw, circle, fill=black] {};
        \node (y) at (0.5,0) [draw, circle] {};
        \node (z) at (1.5,0) [draw, circle] {};
        \node (t) at (2.5,0) [draw, circle, fill=black] {};        
        \draw [-] (u)--(v)--(x)--(y)--(z)--(t)--(v);
  \end{tikzpicture}}
  }\qquad
    \subfigure[]{
    \scalebox{1}{\begin{tikzpicture}[join=bevel,inner sep=0.5mm]
        \node (u) at (0,1) [draw, circle, label={90:$u$}] {};
        \node (v) at (1,1) [draw, circle, label={90:$v$}, fill=black] {};
        \node (x) at (-0.5,0) [draw, circle, fill=black] {};
        \node (y) at (0.5,0) [draw, circle] {};
        \node (z) at (1.5,0) [draw, circle, fill=black] {};
        \node (t) at (2.5,0) [draw, circle] {};        
        \draw [-] (u)--(v)--(x)--(y)--(z)--(t) (v)--(z);
  \end{tikzpicture}}
  }\qquad
    \subfigure[]{
    \scalebox{1}{\begin{tikzpicture}[join=bevel,inner sep=0.5mm]
        \node (u) at (0,1) [draw, circle, label={90:$u$}, fill=black] {};
        \node (v) at (1,1) [draw, circle, label={90:$v$}, fill=black] {};
        \node (x) at (-0.5,0) [draw, circle, fill=black] {};
        \node (y) at (0.5,0) [draw, circle, fill=black] {};
        \node (z) at (1.5,0) [draw, circle, fill=black] {};
        \node (t) at (2.5,0) [draw, circle] {};        
        \draw [-] (u)--(v)--(y) (x)--(y)--(z)--(t) (v)--(z);
  \end{tikzpicture}}
  }\qquad
    \subfigure[]{
    \scalebox{1}{\begin{tikzpicture}[join=bevel,inner sep=0.5mm]
        \node (u) at (0,1) [draw, circle, label={90:$u$}] {};
        \node (v) at (1,1) [draw, circle, label={90:$v$}, fill=black] {};
        \node (x) at (-0.5,0) [draw, circle, fill=black] {};
        \node (y) at (0.5,0) [draw, circle, fill=black] {};
        \node (z) at (1.5,0) [draw, circle, fill=black] {};
        \node (t) at (2.5,0) [draw, circle] {};        
        \draw [-] (u)--(v)--(x)--(y)--(z)--(t)--(v)--(z);
  \end{tikzpicture}}
  }\qquad
    \subfigure[]{
    \scalebox{1}{\begin{tikzpicture}[join=bevel,inner sep=0.5mm]
        \node (u) at (0,1) [draw, circle, label={90:$u$}, fill=black] {};
        \node (v) at (1,1) [draw, circle, label={90:$v$}, fill=black] {};
        \node (x) at (-0.5,0) [draw, circle, fill=black] {};
        \node (y) at (0.5,0) [draw, circle, fill=black] {};
        \node (z) at (1.5,0) [draw, circle, fill=black] {};
        \node (t) at (2.5,0) [draw, circle] {};        
        \draw [-] (u)--(v)--(x)--(y)--(z)--(t)--(v)--(y) (v)--(z);
  \end{tikzpicture}}
  }\qquad
      \subfigure[]{
    \scalebox{1}{\begin{tikzpicture}[join=bevel,inner sep=0.5mm]
        \node (u) at (0,1) [draw, circle, label={90:$u$}] {};
        \node (v) at (1,1) [draw, circle, label={90:$v$}, fill=black] {};
        \node (x) at (-0.5,0) [draw, circle] {};
        \node (y) at (0.5,0) [draw, circle, fill=black] {};
        \node (z) at (1.5,0) [draw, circle, fill=black] {};
        \node (t) at (2.5,0) [draw, circle] {};        
        \draw [-] (u)--(v) (x)--(y)--(z)--(t) (v)--(y);
        \draw [-] (x) .. controls +(1,-0.5) and +(-1,-0.5) .. (t);
  \end{tikzpicture}}
  }\qquad
      \subfigure[]{
    \scalebox{1}{\begin{tikzpicture}[join=bevel,inner sep=0.5mm]
        \node (u) at (0,1) [draw, circle, label={90:$u$}] {};
        \node (v) at (1,1) [draw, circle, label={90:$v$}, fill=black] {};
        \node (x) at (-0.5,0) [draw, circle, fill=black] {};
        \node (y) at (0.5,0) [draw, circle, fill=black] {};
        \node (z) at (1.5,0) [draw, circle] {};
        \node (t) at (2.5,0) [draw, circle, fill=black] {};        
        \draw [-] (u)--(v) (x)--(y)--(z)--(t) (z)--(v)--(y);
        \draw [-] (x) .. controls +(1,-0.5) and +(-1,-0.5) .. (t);
  \end{tikzpicture}}
  }\qquad
        \subfigure[]{
    \scalebox{1}{\begin{tikzpicture}[join=bevel,inner sep=0.5mm]
        \node (u) at (0,1) [draw, circle, label={90:$u$}] {};
        \node (v) at (1,1) [draw, circle, label={90:$v$}, fill=black] {};
        \node (x) at (-0.5,0) [draw, circle, fill=black] {};
        \node (y) at (0.5,0) [draw, circle] {};
        \node (z) at (1.5,0) [draw, circle, fill=black] {};
        \node (t) at (2.5,0) [draw, circle] {};        
        \draw [-] (u)--(v)--(x)--(y)--(z)--(t) (v)--(z);
        \draw [-] (x) .. controls +(1,-0.5) and +(-1,-0.5) .. (t);
  \end{tikzpicture}}
  }\qquad
        \subfigure[]{
    \scalebox{1}{\begin{tikzpicture}[join=bevel,inner sep=0.5mm]
        \node (u) at (0,1) [draw, circle, label={90:$u$}, fill=black] {};
        \node (v) at (1,1) [draw, circle, label={90:$v$}, fill=black] {};
        \node (x) at (-0.5,0) [draw, circle] {};
        \node (y) at (0.5,0) [draw, circle] {};
        \node (z) at (1.5,0) [draw, circle, fill=black] {};
        \node (t) at (2.5,0) [draw, circle, fill=black] {};        
        \draw [-] (u)--(v)--(x)--(y)--(z)--(t) (z)--(v)--(y);
        \draw [-] (x) .. controls +(1,-0.5) and +(-1,-0.5) .. (t);
  \end{tikzpicture}}
  }\qquad
        \subfigure[]{
    \scalebox{1}{\begin{tikzpicture}[join=bevel,inner sep=0.5mm]
        \node (u) at (0,1) [draw, circle, label={90:$u$}] {};
        \node (v) at (1,1) [draw, circle, label={90:$v$}, fill=black] {};
        \node (x) at (-0.5,0) [draw, circle, fill=black] {};
        \node (y) at (0.5,0) [draw, circle, fill=black] {};
        \node (z) at (1.5,0) [draw, circle, fill=black] {};
        \node (t) at (2.5,0) [draw, circle, fill=black] {};        
        \draw [-] (u)--(v)--(x)--(y)--(z)--(t)--(v)--(y) (v)--(z);
        \draw [-] (x) .. controls +(1,-0.5) and +(-1,-0.5) .. (t);
  \end{tikzpicture}}
  }\qquad 
          \subfigure[]{
    \scalebox{1}{\begin{tikzpicture}[join=bevel,inner sep=0.5mm]
        \node (u) at (0,1) [draw, circle, label={90:$u$}] {};
        \node (v) at (1,1) [draw, circle, label={90:$v$}, fill=black] {};
        \node (x) at (-0.5,0) [draw, circle, fill=black] {};
        \node (y) at (0.5,0) [draw, circle, fill=black] {};
        \node (z) at (1.5,0) [draw, circle, fill=black] {};
        \node (t) at (2.5,0) [draw, circle] {};     
        \draw [-] (u)--(v)--(x)--(y)--(z)(v)--(z) ;
        \draw [-] (y) .. controls +(1,-0.3) and +(-1,-0.3) .. (t);
  \end{tikzpicture}}
  }\qquad
          \subfigure[]{
    \scalebox{1}{\begin{tikzpicture}[join=bevel,inner sep=0.5mm]
        \node (u) at (0,1) [draw, circle, label={90:$u$}] {};
        \node (v) at (1,1) [draw, circle, label={90:$v$}, fill=black] {};
        \node (x) at (-0.5,0) [draw, circle, fill=black] {};
        \node (y) at (0.5,0) [draw, circle, fill=black] {};
        \node (z) at (1.5,0) [draw, circle, fill=black] {};
        \node (t) at (2.5,0) [draw, circle] {};     
        \draw [-] (u)--(v)--(x)--(y)--(z) (t)--(v)--(z);
        \draw [-] (y) .. controls +(1,-0.3) and +(-1,-0.3) .. (t);
  \end{tikzpicture}}
  }\qquad
          \subfigure[]{
    \scalebox{1}{\begin{tikzpicture}[join=bevel,inner sep=0.5mm]
        \node (u) at (0,1) [draw, circle, label={90:$u$}, fill=black] {};
        \node (v) at (1,1) [draw, circle, label={90:$v$}, fill=black] {};
        \node (x) at (-0.5,0) [draw, circle] {};
        \node (y) at (0.5,0) [draw, circle, fill=black] {};
        \node (z) at (1.5,0) [draw, circle, fill=black] {};
        \node (t) at (2.5,0) [draw, circle, fill=black] {};    
        \draw [-] (u)--(v)--(x)--(y)--(z) (z)--(v)--(y);
        \draw [-] (y) .. controls +(1,-0.3) and +(-1,-0.3) .. (t);
  \end{tikzpicture}}
  }\qquad
          \subfigure[]{
    \scalebox{1}{\begin{tikzpicture}[join=bevel,inner sep=0.5mm]
        \node (u) at (0,1) [draw, circle, label={90:$u$}, fill=black] {};
        \node (v) at (1,1) [draw, circle, label={90:$v$}, fill=black] {};
        \node (x) at (-0.5,0) [draw, circle, fill=black] {};
        \node (y) at (0.5,0) [draw, circle] {};
        \node (z) at (1.5,0) [draw, circle, fill=black] {};
        \node (t) at (2.5,0) [draw, circle] {};    
        \draw [-] (u)--(v)--(x)--(y)--(z) (z)--(v)--(y) (v)--(t);
        \draw [-] (y) .. controls +(1,-0.3) and +(-1,-0.3) .. (t);
  \end{tikzpicture}}
  }
\end{center}
  \caption{The possibilities for graph $G$ when $G-u-v$ is isomorphic to $P_4$, $C_4$ or $K_{1,3}$, in the proof of Theorem~\ref{The:IDtn-1Char}. The black vertices form total dominating identifying codes.}
\label{fig:smallgraphs-P4-C4}
\end{figure}


By Lemma~\ref{Lem:deletion}, $G_v=G-u-v$ is a connected, identifiable graph with $\IDt(G_v)=n-3$ for any leaf $u$ and adjacent support vertex $v$. Notice that if $G_v$ is a star, then $\ID(G_v)=n-3$ and we have a contradiction with Lemma~\ref{Lem:ExtremalIDt from ID}. Moreover, if $\delta(G_v)\geq2$, then by Lemma~\ref{Lem:n-1minDegree}, $\ID(G_v)\geq n-3$ and we are done by Lemma~\ref{Lem:ExtremalIDt from ID}. Hence, we can assume that $G_v$ is not a star and has a vertex of degree~1. Then, since $G_v$ is also not a $P_4$ since $n\geq 7$, we have $\ID(G_v)\leq n-4$ by Theorem~\ref{The:IDn-1}. Thus, by induction, $G_v$ has the claimed structure of part~(ii) of the statement. That is, there exist graphs $G''\in \mathcal{A}\cup(\mathcal{A}\bowtie K_1)$ and $G'=G''\bowtie K_m$, for $m\geq1$,  such that we can form the graph $G_v$ by adding a leaf to every vertex in the clique $K_m$ of graph $G'$.

We claim that the only way to add vertex $v$ to $G_v$ is by making it a universal vertex in $G'$ and adding no edges between $v$ and the leaves of $G_v$.

We first show that there can be no edges in $G$ between $v$ and the set $L(G_v)$. Suppose on the contrary that there exists an edge between $w'$ and $v$, where $w'\in L(G_v)$ and $w\in S(G_v)$ is the support vertex adjacent to $w'$. Due to issues with total domination, we first consider the case where $N(v)=\{w',u\}$. By Lemma~\ref{Lem:deletion} we have $\IDt(G_v-w-w')\geq n-5$. Let $C'$ be a total dominating identifying code in $G_v-w-w'$. Now $C'\cup\{u,v,w'\}$ is a total dominating identifying code in $G$ of cardinality $n-2$, a contradiction. 

We then consider the case where $v$ has at least three neighbours in $G$. We split this case based on whether there exists a universal vertex $y$ of $G'$ such that $y\not\in S(G_v)$ (such a vertex exists only if $G''\in\mathcal A\bowtie K_1$). Assume first that such vertex $y$ does not exist. In this case, we may consider the code $C=V(G)\setminus\{w,w'\}$. It is total dominating since $v$ has at least three neighbours. Moreover, $V(G_v)\setminus \{w'\}$ is an identifying code in $G_v$ and we only need codeword $w$ to dominate $w'$. When we consider $C$ and graph $G$, we notice that $u$ and $v$ clearly have unique $I$-sets. Moreover, $w$ is the only non-codeword which is a universal vertex in $G'$ and every other vertex universal in $G'$ has an adjacent leaf codeword (since $y$ does not exist). Thus, $w$ is separated from other vertices. Finally, $w'$ is the only vertex which has exactly $v$ in its $I$-set.

Assume then that the vertex $y\in S(G_v)$ exists. Now, we may consider code $C=V(G)\setminus \{y,w\}$. Code $C$ is clearly total dominating. Moreover, $V(G)\setminus \{w\}$ and $V(G)\setminus \{y\}$  are identifying  codes in $G_v$ and $w$ is only needed to total dominate $w'$ in $G_v$. Since $y$ and $w$ are universal vertices in $G'$, they do not separate anything in $G'$. Moreover, $u$ and $v$ are clearly separated by $C$. Thus, the code is total dominating and identifying. Hence, we may from now on assume that $v$ is not adjacent to any leaf of $G_v$, that is, $L(G)=L(G_v)\cup\{u\}$.

We consider the case where there are some non-edges between the clique $K_m$ in $G'$, and $v$. Let us not have edge $vw$ in graph $G$, where $w$ is some vertex in the clique $K_m$ of $G'$ and $w'$ is either the universal vertex in $G''$, if such a vertex exists, and otherwise the leaf adjacent to $w$. Now, $C=V(G)\setminus\{w',u\}$ is a total dominating identifying code of $G$. Indeed, $V(G_v)\setminus\{w'\}$ is a total dominating identifying code in $G_v$ and $u$ is now the only vertex which has exactly $v$ in its $I$-set while $v$ is separated from all other vertices since it is the only non-leaf vertex which is not adjacent to $w$. 
Hence, we may from now on assume that $v$ has an edge with each vertex of clique $K_m$ in $G'$.

Finally, we are left with the case where we have some non-edges between $v$ and $G''$. Assume first that there does not exist a vertex $z\in V(G'')$ such that $N_G[z]\cup\{u\}=N_G[v]$. Now, we may consider code $C=V(G)\setminus \{u,w'\}$ where $w'$ is either the universal vertex in $G''$, if such a vertex exists, and otherwise some leaf in $L(G)$ other than $u$. Code $C$ is total dominating. Moreover, $V(G_v)\setminus \{w'\}$ is an identifying code in $G_v$. Finally, $u$ has a unique $I$-set as the only vertex adjacent to only $v$, while $v$ is separated from every other vertex in clique $K_m$ by having some non-edge to $G''$ and it is separated from every vertex in $G''$ since $z$ does not exist. Thus, $C$ is a total dominating identifying code of size $n-2$, a contradiction.

Consider then the case where we have the vertex $z\in V(G'')$ with $N_G[z]\cup\{u\}=N_G[v]$. Notice that since $v$ is not a universal vertex to $G''$, neither is $z$. Thus, we may consider $C=V(G)\setminus\{z,w'\}$ where $w'$ is either the universal vertex in $G''$, if such a vertex exists, and otherwise some leaf in $L(G)$ other than $u$. Again, $C$ is clearly total dominating and $V(G_v)\setminus \{w'\}$ is an identifying code in $G_v$. Again, $u$ and $v$ have unique $I$-sets with the same arguments as before. Moreover, any pair of vertices separated by $z$ in $G_v$, is now separated by $v$ in $G$ and hence, $C$ is an identifying code in $G$. Now, we have exhausted all the possibilities and the claim follows.
\end{proof}

\section{An upper bound for graphs of girth at least~5}\label{sec:twinfree}

Notice that the extremal graphs from the previous section either have many twins (stars for example), or small cycles. In this section, we prove a (tight) upper bound for total identifying codes of twin-free graphs of girth at least~5 that is much smaller than the one for the general case. Similar upper bounds for twin-free graphs have been studied in the context of location-domination, see~\cite{FH16,Heia,GGM14} and usual identifying codes~\cite{BFH15, ourIDpaper}.

We will need the following lemma, whose proof was given in~\cite{HHH06} (note that it was extended to a larger graph class in~\cite{ourIDpaper}, that includes all identifiable triangle-free graphs).

\begin{lemma}[\cite{HHH06}]\label{LemSupportBip}
If $T$ is a tree on $n\geq4$ vertices that is not the path $P_4$, then $$\IDt(T)\leq n-s(T).$$ 
\end{lemma}

Lemma~\ref{LemSupportBip} was shown to be tight in~\cite{ourIDpaper}, for example for the 3-corona of any graph, the 1-corona of any triangle-free graph of order at least~3, or any star of order at least~3.


For total dominating identifying codes in trees, the following upper bound is known.

\begin{theorem}[\protect{\cite[Theorem 14]{NLG16}}]\label{LDt-TreeBounds}
If $T$ is a tree on $n\geq3$ vertices, then $\IDt(T)\leq \frac{3(n+\ell(T))}{5}$.
\end{theorem}

This upper bound, together with Lemma~\ref{LemSupportBip}, yields the following corollary. 

\begin{corollary}\label{cor:IDtBound}
If $T$ is a twin-free tree on at least $n\geq 3$ vertices, then $$\IDt(T)\leq 3n/4.$$
\end{corollary}
\begin{proof}
Since $G$ is twin-free, we have $s(T)=\ell(T)$, thus by Lemma~\ref{LemSupportBip} we have $\ID(T)\leq n-\ell(T)$. Thus, if $\ell(T)\geq\frac{n}{4}$, we are done. On the other hand, if $\ell(T)<\frac{n}{4}$, by Theorem~\ref{LDt-TreeBounds}, we have $\ID(T)\leq 3(n+\ell(T))/5<\frac{3n}{4}$. 
\end{proof}

Observe that we have $\IDt(C_6)=4>(3\cdot6)/5$ and hence, one cannot generalize the bound $\IDt(T)\leq3(n+\ell(T))/5$ to a class of twin-free graphs including 6-cycles. However, we can generalize Corollary~\ref{cor:IDtBound} to all twin-free graphs of girth at least~5 by finding a small total dominating identifying code in a well chosen sub-tree.

\begin{theorem}\label{the:IDtbipart}
If $G$ is a connected twin-free graph of girth at least $5$ on $n\geq 3$ vertices, then $$\IDt(G)\leq 3n/4.$$
\end{theorem}
\begin{proof}
Observe first that if $G$ has a twin-free spanning tree $T$, then $T$ has a total-dominating identifying code of size at most $3n/4$ by Corollary~\ref{cor:IDtBound}. Moreover, since $G$ does not have any triangles or $4$-cycles, one can check that the same code is also total dominating identifying in $G$. 

Assume that every spanning tree of $G$ has some twins, that is, leaves with the same adjacent support vertex, and assume that $T$ is the spanning tree with the least amount of twins among all spanning trees of $G$. Now, for each support vertex, we remove all but one adjacent leaf and we denote by $T'$ the resulting twin-free tree and say that it has $n'$ vertices. Notice that $n'\geq4$. Indeed, if $n'\leq 3$, then $T$ does not contain a $P_4$ as a subgraph. Since $T$ is connected, it is a star. However because $n\geq4$ and $G$ is twin- and triangle-free, we get a contradiction. Thus, $n'\geq4$. Now, $\IDt(T')\leq 3n'/4$ by Corollary~\ref{cor:IDtBound}. Let $C$ be a total dominating identifying code in $T'$ of at most size $3n'/4$; observe that $S(T')\subseteq C$, since each leaf needs a neighbour in $C$ in order to be totally dominated.

If $v\in S(T)$ is a support vertex which has $u,w\in N_T(v)\cap L(T)$, then we have removed either $u$ or $w$ to form $T'$. However, since $G$ is twin-free, $u$ or $w$, say $u$, has some other neighbours in $G$. Consider $x\in N_G(u)$. Observe that since $T$ had the minimal number of twins among all the spanning trees, we have $x\in S(T')$. Indeed, otherwise $T-uv+xu$ would have at least one twin less than $T$, a contradiction. However, then, $x\in C$ and hence, $|I_G(C;u)|\geq2$ and because $G$ has no 4-cycles, $u$ is uniquely distinguished. Thus, $C$ is a total dominating identifying code of $G$ with cardinality at most $3n'/4<3n/4$.
\end{proof}

Note that Theorem~\ref{the:IDtbipart} cannot hold for graphs that contain twins (because of complete bipartite graphs, for which the total identifying code number is $n-2$ or $n-1$) or triangles (because of complements of half-graphs, for which the (total dominating) identifying code number is $n-1$~\cite{FGKNPV11} as seen in Proposition~\ref{prop:extremal-family-IDtn-1}). However, in the following corollary, we give a generalized form for all connected graphs of girth at least $5$.

\begin{corollary}\label{cor:g5Generalized}
If $G$ is a connected graph of girth at least $5$ on $n\geq 3$ vertices, then $$\IDt(G)\leq (3n+\ell(G)-s(G))/4.$$
\end{corollary}
\begin{proof}
Let $G$ be a connected graph of girth at least $5$ on $n\geq 3$ vertices. If $G$ is a star, then $(3n+\ell(G)-s(G))/4=n-1/2$ and the claim holds. Assume then that $G$ is not a star. Notice that if we have any twins, then they are leaves with the same adjacent support vertex. Denote by $G'$ the graph obtained from $G$ by removing leaves until $G'$ is twin-free and let $G'$ have order $n'$. Since $G$ is not a star, we have $n'\geq3$ and thus by Theorem~\ref{cor:g5Generalized}, $\IDt(G')\leq 3n'/4$. Let $C'$ be an optimal total dominating identifying code in $G'$. We have $S(G')\subseteq C'$. Thus, we may construct a total dominating identifying code $C$ for $G$ as $C=C'\cup(L(G)\setminus L(G'))$. We have $|C|\leq 3n'/4+(\ell(G)-\ell(G'))=3(n-\ell(G)+s(G))/4+(\ell(G)-s(G))=(3n+\ell(G)-s(G))/4$.\end{proof}

\begin{remark}
Theorem \ref{the:IDtbipart} improves the known upper bound for usual identifying codes in connected twin-free graphs of girth at least $5$ when $\ell(G)> n/8$. Indeed, the current best known upper bound for such graphs is $\ID(G)\leq (5n+2\ell(G))/7$, \cite{ourIDpaper}. When $\ell(G)\geq n/8$, we have $(5n+2\ell(G))/7\geq 3n/4$. Furthermore, Corollary~\ref{cor:g5Generalized} improves the bound when $s(G)>n/(a+7)$ where $\ell(G)=a\cdot s(G)$ and $a\geq1$ is a constant. 
\end{remark}

\medskip

Consider now some graphs which actually attain the $\frac{3n}{4}$ upper bound. In~\cite{NLG16}, the authors have shown that if  $\IDt(T)=3(n+\ell(T))/5$, then $T\in \mathcal{T}$, where $\mathcal{T}$ is defined with the following iterative process. Let $T_0=P_8$ and let there exist four different \emph{statuses} of vertices, $A$, $B$, $C$ and $D$, denoted by $s(v)$ for vertex $v$. For $T_0$, leaves have status $C$, support vertices status $A$, non-leaf vertices adjacent to support vertices status $B$ and the remaining two vertices have status $D$. Now, we can create a tree $T_i$ from a tree $T_{i-1}\in \mathcal{T}$ by applying either of two operations $\phi_1$ or $\phi_2$. 

In operation $\phi_1$, we add a path $P_5$ to $T_{i-1}$, with vertices $y,z,u,v,w$, where the consecutive vertices have an edge between them, with an edge between $y$ and any vertex $x\in V(T_{i-1})$ with $s(x)=C$. Moreover, we have statuses $s(y)=D$, $s(z)=D$, $s(u)=B$, $s(v)=A$ and $s(w)=C$.

In operation $\phi_2$, we add  path $P_4$ to $T_{i-1}$, with vertices $y,z,u,v$, where the consecutive vertices have an edge between them, with an edge between $y$ and any vertex $x\in V(T_{i-1})$ with $s(x)=D$. Moreover, we have statuses $s(y)=D$, $s(z)=B$, $s(u)=A$ and $s(v)=C$.

For a graph $H$, the \emph{$3$-corona} of $H$ is the graph of order~$4|V(H)|$ obtained from $H$ by adding a vertex-disjoint copy of a path $P_3$ for each vertex $v$ of $H$ and adding an edge joining $v$ to one end of the added path (see~\cite{ourIDpaper} and~\cite[Section 1.3]{bookTD}).

Since a twin-free tree $T$ on $n\geq8$ vertices can attain the upper bound in Corollary~\ref{cor:IDtBound} only when $T\in \mathcal{T}$ and $s(T)=n/4$, we can construct $T$ from $T_0$ by iteratively applying operation $\phi_2$. Moreover, this is equivalent with saying that $T$ is the $3$-corona of some tree $H$ on at least two vertices where $s(v)=D$ if $v\in V(H)$. This leads to the following theorem (noticing that the path $P_4$ is also an example).


\begin{theorem}\label{the:3coronas}
If $T$ is a twin-free tree on $n\geq4$ vertices with $\IDt(T)=3n/4$, then $T$ is the $3$-corona of some tree $H$.
\end{theorem}

Observe that we cannot generalize Theorem~\ref{the:3coronas} to all twin-free graphs of girth at least~5, since $\IDt(C_8)=6$, but we do not know if there exist other counterexamples. However, we can deduce from the proof of Theorem~\ref{the:IDtbipart} that if some other counterexample exists, then that graph has only $3$-coronas as its twin-free spanning trees.

\section{Bounds between related parameters}\label{sec:bounds-related}

In this section we prove bounds relating the parameter $\IDt$ to similar parameters. Tight bounds relating the parameters $\ID$, $\LD$ and $\OLD$ were provided in the literature. It was indeed proved in~\cite{GKM08} that for any identifiable graph $G$, $\ID(G)\leq 2\LD(G)$ holds (and is tight). Similar bounds were proved in the PhD thesis~\cite[Chapter~2.4.1]{SewellPhD}, showing that $\ID(G)\leq 2\OLD(G)$, $\OLD(G)\leq 2\LD(G)$ and $\OLD(G)\leq 2\ID(G)$, and providing tight families of examples for each bound. As we will see, we can also bound $\IDt(G)$ by a constant times $\LD(G)$, $\LDt(G)$ and $\ID(G)$, but not exactly by a factor of $2$ like in the other bounds. 
We have presented some relationships between these types of codes in Figure~\ref{Dominating codes}. Thus, we have $\IDt(G)\leq \IDe(G)$ and $\IDt(G)\leq \IDs(G)$. As we will see, we cannot get similar constant type bounds for these parameters.

\subsection{Relation with (classic) identifying codes}

\begin{theorem}\label{IDt<2ID-2}
Let $G$ be a connected graph with $\ID(G)\geq3$, then $$\IDt(G)\leq2\ID(G)-2.$$
\end{theorem}
\begin{proof}
Assume that $C$ is an optimal identifying code in $G$ with cardinality at least~3. Since $G$ is connected and $C$ is identifying, if $I(c)=\{c\}$ for some codeword $c\in C$, then we may add any adjacent non-codeword to $C$ and vertex $c$ becomes totally dominated. Since at least the first non-codeword we add to the code can be chosen to connect two codewords (indeed the non-codeword cannot be only dominated by the codeword), we immediately get that we need at most $|C|-1$ vertices in this process, and so $\IDt(G)\leq 2\ID(G)-1$. 

Observe that if any two codewords are adjacent in $C$, then there is also a third codeword adjacent to one of them, that distinguishes them. Thus, we need at most $|C|-3+|C|=2|C|-3$ codewords to form a total dominating identifying code and we are done. Hence, we assume from now on that every optimal identifying code in $G$ has only isolated codewords. Therefore, each non-codeword has at least two adjacent codewords. Observe that if any non-codeword $x$ has three or more adjacent codewords, then we are done by adding vertex $x$ to the code and then proceeding as in the first step. Likewise, if a pair of non-codewords together has four or more distinct adjacent codewords, we can proceed similarly. Thus, we may now assume that this does not occur.

Moreover, we may assume that $2|C|\leq n$. Otherwise, the claim follows from Proposition~\ref{Prop:IDtn-1}. Now, consider the bipartite graph $B$ obtained from $G$ by keeping only the edges between the codewords and non-codewords. If we contract non-codewords into edges (recall that each non-codeword has degree~2 in $B$) of $B$ to obtain a graph $G'$, since $2|C|\leq n$, we notice that we have $|C|$ vertices and at least $|C|$ edges in $G'$. Moreover, graph $B$ could not have a $4$-cycle since $C$ is an identifying code and hence, we do not contract non-codewords into parallel edges. Thus, the resulting graph $G'$ has a cycle of length at least~3. If this cycle has at least four vertices, then there existed two non-codewords in $G$ with at least two distinct adjacent codewords each (this corresponds to a matching of size~2 in $G'$), a contradiction, and the claim follows. The same is true, if the cycle is a triangle and there exist any other vertex in $G'$ (then here also, $G'$ contains a matching of size~2). 

Finally, we are left with the case where $\ID(G)=3$. Moreover, by the previous argumentation, the only case we need to consider is the one where each codeword has degree~2, is adjacent to exactly two non-codewords, and there is a total of six vertices. If none of the non-codewords are adjacent in $G$, then $G$ is the cycle $C_6$ which has $\IDt(C_6)=4$, and we are done. Thus, we may assume that there is an edge between some non-codewords. However, now we can find an induced path $P_4$ $c_1,u,v,c_2$ which starts with $c_1\in C$, has non-codewords $u$ and $v$ as the middle vertices and ends with $c_2\in C$. The path is induced due to the properties of codeword vertices. We claim that these four vertices form a total dominating identifying code $C'$. Observe that the single vertex $w$ which belongs to neither $C$ nor $C'$ is the only vertex with $I$-set $\{c_1,c_2\}$ and the single vertex in $C\setminus C'$ is the only vertex which is not adjacent to either codeword in $C\cap C'$. Finally, the vertices in $C'$ are pairwise separated since $G[C']$ is a 4-path. 
\end{proof}

The upper bound of Theorem~\ref{IDt<2ID-2} is tight for $1$-coronas of complete graphs from which we remove a single leaf.

\subsection{Relation with locating-total dominating sets}

\begin{theorem}\label{thm:bund-IDtLDt}
If $G$ is a connected identifiable graph on at least three vertices, then
$$\LDt(G)\leq \IDt(G)\leq 2\LDt(G).$$
\end{theorem}
\begin{proof}
In \cite[Proof of Theorem $8$]{GM07}, the authors have shown that if $D'$ is a locating-dominating set in $G$, then there exists an identifying code $D$ such that $D'\subseteq D$ and $|D|\leq 2|D'|$. 

Assume now that $C'$ is an optimal locating-total dominating set in graph $G$. Thus, $C'$ is a locating-dominating set and hence, there exists an identifying code $C$ such that $C'\subseteq C$ of cardinality $ |C|\leq 2|C'|=2\LDt(G)$. Moreover, since $C'$ is total dominating, also $C$ is total dominating and hence, $\IDt(G)\leq 2\LDt(G)$ as we claimed.
\end{proof}

The upper bound from Theorem~\ref{thm:bund-IDtLDt} is tight for complete graphs of odd order from which we have removed a maximal matching, indeed for such a graph $G$ of order $n=2k+1$ we have $\IDt(G)=n-1=2k$ by Proposition~\ref{prop:extremal-family-IDtn-1} but $\LDt(G)=k$ (for every edge of the removed matching, the two endpoints are twins in $G$, so one of them must belong to any locating-dominating set and $\LDt(G)\geq k$; on the other hand, selecting one such vertex for each pair gives a locating-total dominating set of size $k$).

\subsection{Relation with locating-dominating sets}

We first relate the locating-total domination number with the usual location-domination number.

\begin{theorem}\label{thm:bound-LDt-LD}
If $G$ is a connected graph on at least three vertices, then
$$\LD(G)\leq \LDt(G)\leq 2\LD(G)-1.$$
\end{theorem}
\begin{proof}
Let $G$ be a connected graph and let $C$ be an optimal locating-dominating set in $G$. We can create a locating-total dominating set from $C$ by adding a codeword adjacent to each codeword in $C$. Thus, $\LDt(G)\leq 2\LD(G)$. However, since $G$ is connected, there exists a non-codeword $u$ with $|I(u)|\geq2$ if $\LD(G)\geq3$. Thus, we may add $u$ and $\LD(G)-2$ other vertices to the code and we get the claimed upper bound. Notice that $\LD(G)\geq2$, since we have at least three vertices in $G$. Moreover, if $\LD(G)=2$ and we have $|I(u)|=1$ for each vertex in $V(G)$, then $G$ is a path on four vertices. However, $\LDt(P_4)\leq \IDt(P_4)= 3$ and the claim follows.
\end{proof}

The upper bound from Theorem~\ref{thm:bound-LDt-LD} is tight for stars which have all but one of their edges subdivided once. Indeed, for such a tree $T_k$ of order $n=3k+2$, we have $\LD(T_k)=k+1$ and $\LDt(T_k)=2k+1$. For each leaf of $T_k$, either the leaf or its support vertex must be in any dominating set to dominate the leaf,  so $\LD(T_k)\geq k+1$. Moreover, to get a total dominating set, we need two vertices in each branch which has three vertices and the central vertex, so $\LDt(T_k)\geq 2k+1$. On the other hand, taking every support vertex gives a locating-dominating set of size $k+1$. Taking every support vertex together with its degree~2 neighbour gives a locating-total dominating set of size $2k+1$.

\medskip

Notice that Theorems \ref{IDt<2ID-2}, \ref{thm:bund-IDtLDt} and \ref{thm:bound-LDt-LD} together with \cite[Theorem $8$]{GM07} and Figure \ref{Dominating codes} imply that $\IDt(G)\leq 2\ID(G)-2\leq 4\LD(G)-2 $ and that  $\IDt(G)\leq 2\LDt(G)\leq 4\LD(G)-2 $. However, as we can see in the following theorem, this bound is not tight.

\begin{theorem}\label{IDt<3LD-log}
If $G$ is a connected identifiable graph with $\LD(G)\geq 2$, then $$\IDt(G)\leq3\LD(G)-\log_2(\LD(G)+1).$$
\end{theorem}
\begin{proof}
Let $C_{LD}$ be an optimal locating-dominating set with at least two codewords in an identifiable connected graph $G$. We have $\ID(G)\leq 2|C_{LD}|$ by \cite[Theorem $8$]{GKM08}. Moreover, following the proof of~\cite{GKM08}, we may construct an identifying code from $C_{LD}$ by just adding at most $|C_{LD}|$ additional vertices to $C_{LD}$. Denote by $C_A$ a smallest set of vertices which we can add to $C_{LD}$ so that $C_A\cup C_{LD}=C_{ID}$ is an identifying code. Observe that every vertex in $C_A$ is adjacent to a vertex in $C_{LD}$ (since every vertex not in $C_{LD}$ is adjacent to a vertex in $C_{LD}$). Moreover, when we add the $|C_A|$ new codewords, those new codewords are also total dominating some (old) codewords of $C_{LD}$. We denote the dominated codewords by $C_D\subseteq C_{LD}$. In particular, we have $2^{|C_D|}-1\geq|C_A|$ since the vertices in $C_A$ were all dominated and separated among each other by the vertices of $C_D$. 

Therefore, to make $C_{ID}$ total dominating, it suffices to add only codewords which dominate vertices in $C_{LD}\setminus C_D$. Hence, we can build a total dominating identifying code of cardinality at most $|C_{LD}|+(|C_{LD}|-|C_D|)+|C_A|\leq 2|C_{LD}|+|C_A|-\log_2(|C_A|+1)\leq 3|C_{LD}|-\log_2(|C_{LD}|+1).$
\end{proof}

We can show that the bound of Theorem~\ref{IDt<3LD-log} is almost tight, as follows.

\begin{proposition}\label{Prop:IDt<3LD-log-tight}
For every integer $k\geq 2$, there is a connected identifiable graph $G_k$ with $\LD(G_k)=2^k-1$ and $\IDt(G_k)=3\cdot 2^k-2k-3=3\LD(G_k)-2\log_2(\LD(G_k)+1)$.
\end{proposition}
\begin{proof}
We build $G_k$ as follows. $G_k$ contains a set $A=\{a_1,\ldots,a_k\}$ of $k$ vertices. 
For each vertex $a_i$ in $A$, we add a leaf $b_i$ adjacent to $a_i$. Moreover, for each subset $S$ of $A$ of size at least~2 (there are $2^k-k-1$ such sets), we have vertices $x_S$, $x'_S$, $y_S$ and $z_S$ with the following edges. Vertices $x_S$ and $x'_S$ have all the vertices $a_i$ with $a_i\in S$ as neighbours. Moreover, $x_S$ is adjacent to $x'_S$ and to $y_S$. Vertex $z_S$ is adjacent to $x_S$, $x'_S$ and $y_S$. See Figure~\ref{fig:IDt<3LD-log-tight} for an illustration.

To see that $\LD(G_k)\leq 2^k-1$, notice that the set consisting of $A$ and each vertex $y_S$ forms a locating-dominating set. Moreover we need at least $2^k-k-1$ vertices to dominate the vertices of type $y_S$, and $k$ vertices to dominate the vertices of type $b_i$, so $\LD(G_k)\geq 2^k-1$.

Next, observe that each vertex $y_S$ needs to be in any identifying code to separate $x_S$ from $x'_S$, for each subset $S$ of $A$ of size at least~2. We also need one of $z_S$ and $x_S$ to totally dominate $y_S$. Moreover, $x'_S$ must belong to the code in order to separate $y_S$ from $z_S$. Vertex $a_i$ must belong to the code to totally dominate $b_i$, for each $i$ in $\{1,\ldots,k\}$. Thus, $\IDt(G_k)\geq 3(2^k-k-1)+k=3\cdot 2^k-2k-3$. Finally, the set consisting of $A$, each vertex $y_S$, $z_S$, and $x'_S$ forms a total dominating identifying code, thus $\IDt(G_k)\leq 3\cdot 2^k-2k-3$.
\end{proof}

\begin{figure}[!htpb]
\centering
\begin{tikzpicture}

\node[small node,label={0:$x_{S_1}$}](xS1) at (0.5,2)    {};
\node[small node,fill=lightgray,label={180:$x'_{S_1}$}](x'S1) at (-0.5,2)    {};
\node[small node,fill,label={0:$y_{S_1}$}](yS1) at (0.5,3)    {};
\node[small node,fill=lightgray,label={180:$z_{S_1}$}](zS1) at (-0.5,3)    {};

\node[small node,label={0:$x_{S_2}$}](xS2) at (4.5,2)    {};
\node[small node,fill=lightgray,label={180:$x'_{S_2}$}](x'S2) at (3.5,2)    {};
\node[small node,fill,label={0:$y_{S_2}$}](yS2) at (4.5,3)    {};
\node[small node,fill=lightgray,label={180:$z_{S_2}$}](zS2) at (3.5,3)    {};

\node[small node,fill,label={180:$a_1$}](a1) at (0,0)    {};
\node[small node,label={180:$b_1$}](b1) at (0,-1)    {};
\node[small node,fill,label={180:$a_2$}](a2) at (1,0)    {};
\node[small node,label={180:$b_2$}](b2) at (1,-1)    {};
\node[small node,fill,label={180:$a_k$}](ak) at (4,0)    {};
\node[small node,label={180:$b_k$}](bk) at (4,-1)    {};
\node (dots1) at (2.25,-0.75)    {$\ldots$};
\node (dots2) at (2,2.5)    {$\ldots$};
\node[draw,dashed,rectangle,rounded corners,minimum width=5.5cm,minimum height=1cm,line width=0.5mm] (r) at (1.8,0) {};
\node (A) at (5,0)    {$A$};
\node[draw,ellipse,densely dotted,minimum width=2cm,minimum height=0.7cm] (e1) at (1.25,0) {};
\node (A) at (1.5,0)    {\textcolor{gray}{$S_1$}};
\node[draw,ellipse,densely dotted,minimum width=2.5cm,minimum height=0.7cm] (e2) at (3,0) {};
\node (A) at (2.75,0)    {\textcolor{gray}{$S_2$}};

\path[draw,thick]
    (a1) edge node {} (b1)
    (a2) edge node {} (b2)
    (ak) edge node {} (bk)
    (xS1) edge node {} (x'S1)
    (yS1) edge node {} (xS1)
    (yS1) edge node {} (zS1)
    (xS1) edge node {} (zS1)
    (x'S1) edge node {} (zS1)
    (x'S1) edge node {} (0.25,0)
    (xS1) edge node {} (0.25,0)
    (x'S1) edge node {} (2.25,0)
    (xS1) edge node {} (2.25,0)
    (xS2) edge node {} (x'S2)
    (yS2) edge node {} (xS2)
    (yS2) edge node {} (zS2)
    (xS2) edge node {} (zS2)
    (x'S2) edge node {} (zS2)    
    (x'S2) edge node {} (1.75,0)
    (xS2) edge node {} (1.75,0)
    (x'S2) edge node {} (4.25,0)
    (xS2) edge node {} (4.25,0)
    ;

\end{tikzpicture}\centering
\caption{Sketch of the graph $G_k$ built in Proposition~\ref{Prop:IDt<3LD-log-tight}, where $S_1$ and $S_2$ are two subsets of $A$ of size at least~2. The black vertices form an optimal locating-dominating set, while the black and gray vertices together form an optimal total dominating identifying code.}\label{fig:IDt<3LD-log-tight}
\end{figure}

\subsection{Relations with self-identifying and error-correcting identifying codes}

In the following two propositions, we show that there does not exist general bounds of types $\IDs(G)\leq c\IDt(G)$ or $\IDe(G)\leq c\IDt(G)$ for any constant $c$. In fact, the constructions we give, offer (almost) the largest possible gaps between any $\IDt$ and any other parameter. Indeed, if a graph has $2^k-1$ vertices, then there are at least $k$ vertices in any total dominating identifying code in graph $G$. Hence, the values of $\IDe(G)$ or $\IDs(G)$ alone tell almost nothing about the value of $\IDt(G)$.

\begin{proposition}\label{thm:bound-IDt-SID}
Let $k\geq4$ be an even integer. There exists a connected graph $G$ with $\IDt(G)=k$ and $\IDs(G)=2^k-2$.
\end{proposition}
\begin{proof}
Let $k\geq4$ be an even integer. We construct graph $G$ in the following way. We start from a complete graph $K_k$ on vertex set $X=\{x_1,\dots, x_k\}$. After that we create $2^k-k-2$ vertices and join each to a distinct subset $X'\subseteq X$ of vertices of cardinality $1\leq |X'|\leq k-2$. Then, we join any vertex $u$ with $N(u)=X'$ to the single vertex $v$ with $N(v)=X\setminus X'$. If $\deg(u)=1$ and $N(u)=\{x_{2i+1}\}$, then we join it to vertex $v$ with $N(v)=\{x_{2i+2}\}$ where $0\leq i\leq k/2-1$. Finally, we remove a perfect matching $\{x_1x_k,x_2x_3,x_4x_5,\dots, x_{k-2}x_{k-1}\}$ from the vertices within the clique $K_k$.

Observe that $C=X$ forms an optimal total dominating identifying code. Moreover, since no vertex has its closed neighbourhood completely included in another neighbourhood, set $V(G)$ is a self-identifying code. We claim that it is also optimal. Suppose on the contrary that $C$ is a self-identifying code of smaller cardinality. Assume first that some vertex $x_{i}\not\in C$. There exists a vertex $u$ with $N(u)=\{x_{i},v\}$. Now, $I(u)\subseteq\{u,v\}\subseteq I(v)$ and hence, $C$ is not a self-identifying code. Assume then that some vertex $u\not\in X$ is a non-codeword. There is vertex $v\in N(u)\setminus X$. Assume first that $N(u)=\{v,x_i\}$. Let $x_j\in N(v)$. Now $I(u)\subseteq\{x_i,v\}\subseteq I(x_j)$. Moreover, if $N(u)=X'\cup\{v\}$, then there exists a vertex $x_j\in X\setminus X'$ with $X\cup\{v\}\subseteq I(x_j)$, a contradiction. Hence, $V(G)$ is an optimal self-identifying code and the claim follows.  
\end{proof}

In the following proposition, we consider the possible gap between total dominating identifying codes and error-correcting identifying codes.

\begin{proposition}\label{thm:bound-IDt-EID}
Let $k\geq4$ be an integer. There exists a connected graph $G$ with $\IDt(G)=k$ and $\IDe(G)=2^k-1$.
\end{proposition}
\begin{proof}
Let $k\geq4$ be an integer. We construct graph $G$ in the following way. We start from a path $P_k$ on vertex set $X=\{x_1,\dots, x_k\}$. After that we create a set $Y$ of $2^k-k-1$ vertices and join each to a distinct nonempty subset $X'\subseteq X$ of vertices such that $N[x_i]\cap X\neq X'$ for any $x_i$. After this, we add edges between the vertex $w$ of $Y$ joined to all vertices of $X$, and each vertex of $Y$ joined to the vertices of some set $X'\subseteq X$ with $|X'|=1$.

Observe that code $C=X$ is an optimal total dominating identifying code in $G$. Moreover, we claim that $V(G)$ is an optimal error-correcting identifying code in $G$. Let $C$ be an optimal error-correcting identifying code in $G$. Notice that for each $i$ ($1\leq i\leq k$), $x_i\in C$. Indeed, otherwise some vertex $u$ in $Y$ with $N(u)=\{x_i,x_j\}$ would have $|I(u)|\leq2$, a contradiction. Moreover, if a vertex $u$ with $|N(u)\cap X|\leq2$ is a non-codeword, then $|I(u)|\leq2$, a contradiction. Similarly, if the vertex $w$ of $Y$ with $N(w)=X$ is a non-codeword, then each vertex $u$ of $Y$ with $|N(u)\cap X|=1$ has $|I(u)|\leq2$. Finally, suppose by contradiction that some vertex $v$ of $Y$ with $|N(v)\cap X|\geq 3$ is a non-codeword. Then, we can find a vertex $x_i$ of $N(v)\cap X$ such that there exists a vertex $v'$ of $Y$ with $N(v')\cap X=(N(v)\cap X)\setminus\{x_i\}$. Then, $I(v)\triangle I(v')\subseteq \{v',x_i\}$, a contradiction. Thus, $V(G)$ is an optimal error-correcting identifying code.
\end{proof}

Observe that the construction in Proposition~\ref{thm:bound-IDt-EID} is best possible since any graph on $2^k-1$ vertices has at least $k$ vertices in any identifying code. 

\section{Concluding remarks}\label{sec:conclu}

We have characterized the extremal graphs for total dominating identifying codes (that is, those graphs $G$ of order $n$ for which $\IDt(G)\in\{n-1,n\}$), extending the existing characterization for usual identifying codes from~\cite{FGKNPV11}. All these graphs either have twins or cycles of lengths~3 and~4; in the absence of these features, we showed that the graph has a relatively small total dominating identifying code, since $\IDt(G)\leq 3n/4$.

It would be interesting to characterize the graphs for which the $\IDt(G)\leq 3n/4$ upper bound for twin-free graphs of girth at least~5 of Theorem~\ref{the:IDtbipart} is tight, extending the characterization obtained for trees (Theorem~\ref{the:3coronas}). Is $C_8$ the only other tight example besides the $3$-coronas?

Perhaps it is possible to extend the $3n/4$ bound from girth~5 graphs to some twin-free triangle-free graphs with 4-cycles (we need the triangle-free restriction because of complements of half-graphs, and the twin-free restriction because of stars).


We note that several known bounds for twin-free trees are tight only for coronas, like the $3n/4$ bound for total dominating identifying codes of Corollary~\ref{cor:IDtBound}, the $2n/3$ bound for identifying codes from~\cite{ourIDpaper}, the $2n/3$ bound for locating-total dominating sets from~\cite{FH16} and $n/2$ bound for dominating sets~\cite{FJKR85, PX82}. An exception is the $n/2$ upper bound for locating-dominating sets, see~\cite{Heia}, for which the class of trees reaching the bound is more intricate. 

We also introduced multiple tight bounds for $\IDt$ based on other domination parameters. However, in the case of locating-dominating sets, we still have a gap in the logarithmic term between the bound in Theorem~\ref{IDt<3LD-log} and the construction in Proposition~\ref{Prop:IDt<3LD-log-tight}. We have shown that when we do not give any restrictions for the graph structure, then the self-identification number and the error-correcting identification number do not give (almost) any information about the total dominating identification number. However, is it possible to give restrictions for the graph structure so that these values become closer to each other?

\subsection*{Acknowledgements}

Florent Foucaud was financed by the French government IDEX-ISITE initiative 16-IDEX-0001 (CAP 20-25) and by the ANR project GRALMECO (ANR-21-CE48-0004).
Tuo\-mo Lehtil\"a's research was supported by the Finnish Cultural Foundation and by the Academy of Finland grant 338797.

\end{document}